\documentclass{amsart}
\usepackage{amssymb}
\usepackage{graphicx}
\usepackage{amscd}
\usepackage{stmaryrd}

\allowdisplaybreaks
\numberwithin{figure}{section}
\numberwithin{table}{section}
\numberwithin{equation}{section}

\newtheorem{thm}{Theorem}[section]

\newtheorem{lem}[thm]{Lemma}
\newtheorem{cor}[thm]{Corollary}
\newtheorem{remark}{Remark}[section]
\newcommand{\C}{\mathbb{C}}
\newcommand{\N}{\mathbb{N}}
\newcommand{\R}{\mathbb{R}}
\newcommand{\Z}{\mathbb{Z}}

\newcommand{\B}{\mathcal B}

\newcommand{\CalC}{\mathcal C}
\newcommand{\w}{\omega}
\newcommand{\alam}{a}
\newcommand\interior{int}
\newcommand{\Lc}{\mathcal L}
\newcommand{\be}{b}
\newcommand{\wrm}{\mathrm{w}}

\newcommand{\erm}{\mathrm{e}}

\renewcommand\Re{\operatorname{Re}}
\def\cprime{\char"7E }
\DeclareMathOperator{\err}{err}

\DeclareMathOperator{\Lip}{Lip}

\begin{document}

\title[Computation of Hausdorff Dimension]
{$C^m$ Eigenfunctions of Perron-Frobenius Operators and a 
New Approach to Numerical Computation of Hausdorff Dimension:
Applications in $\R^1$}
\author{Richard S. Falk}
\address{Department of Mathematics,
Rutgers University, Piscataway, NJ 08854}
\email{falk@math.rutgers.edu}
\urladdr{http://www.math.rutgers.edu/\char'176falk/}

\author{Roger D. Nussbaum}
\address{Department of Mathematics,
Rutgers University, Piscataway, NJ 08854}
\email{nussbaum@math.rutgers.edu}
\urladdr{http://www.math.rutgers.edu/\char'176nussbaum/}
\thanks{The work of the second author was supported by
NSF grant DMS-1201328.}
\date{August 8, 2017}

\begin{abstract}
We develop a new approach to the computation of the Hausdorff dimension of
  the invariant set of an iterated function system or IFS.  In the one
  dimensional case that we consider here, our methods require only $C^3$
  regularity of the maps in the IFS.  The key idea, which has been known in
  varying degrees of generality for many years, is to associate to the IFS a
  parametrized family of positive, linear, Perron-Frobenius operators
  $L_s$. The operators $L_s$ can typically be studied in many different Banach
  spaces. Here, unlike most of the literature, we study $L_s$ in a Banach
  space of real-valued, $C^k$ functions, $k \ge 2$. We note that $L_s$ is not
  compact, but has essential spectral radius $\rho_s$ strictly less than the
  spectral radius $\lambda_s$ and possesses a strictly positive $C^k$
  eigenfunction $v_s$ with eigenvalue $\lambda_s$.  Under appropriate
  assumptions on the IFS, the Hausdorff dimension of the invariant set of the
  IFS is the value $s=s_*$ for which $\lambda_s =1$.  This eigenvalue problem
  is then approximated by a collocation method using continuous piecewise
  linear functions.  Using the theory of positive linear operators and
  explicit a priori bounds on the derivatives of the strictly positive
  eigenfunction $v_s$, we give rigorous upper and lower bounds for the
  Hausdorff dimension $s_*$, and these bounds converge to $s_*$ as the mesh
  size approaches zero.

\end{abstract}

\maketitle

\date{August 8, 2017}

\section{Introduction}
\label{sec:intro}

Our interest in this paper is in finding rigorous estimates for the Hausdorff
dimension of invariant sets for iterated function systems or IFS's.  The case
of graph directed IFS's (see \cite{I} and \cite{H}) is also of great interest
and can be studied by our methods, but for simplicity we shall restrict
attention here to the IFS case.

Let $D \subset \R$ be a nonempty compact set and $\theta_j: D \to D$, $1 \le j
\le m$, a contraction mapping, i.e., a Lipschitz mapping with Lipschitz
constant $\Lip(\theta_j)$, satisfying $\Lip(\theta_j):= c_j <1$.  If $m <
\infty$ and the above assumption holds, it is known that there exists a
unique, compact, nonempty set $C \subset D$ such that $C = \cup_{j =1}^m
\theta_j(C)$.  The set $C$ is called the invariant set for the IFS $\{\theta_j
\, | \, 1 \le j \le m\}$.

Although we shall eventually specialize, it may be helpful to describe
initially some functional analysis results in the generality of the
previous paragraph. Let $H$ be a bounded, open subset of $\R$, which
is a finite union of open intervals, and let $C^k(\bar H)$ denote the
real Banach space of $C^k$ real-valued maps, all of whose derivatives
of order $\nu \le k$ extend continuously to $\bar H$.  For a given
positive integer $N$, assume that $g_j: \bar H \to (0, \infty)$ are
strictly positive $C^N$ functions for $1 \le j \le m < \infty$ and
$\theta_j: \bar H \to \bar H$, $1 \le j \le m$, are $C^N$ maps and
contractions.  For $s >0$ and integers $k$, $0 \le k \le N$, one can
define a bounded linear map $L_{s,k}: C^k(\bar H) \to C^k(\bar H)$ by
the formula
\begin{equation}
\label{intro1.2}
(L_{s,k} w)(x) = \sum_{j=1}^m [g_j(x)]^s w(\theta_j(x)).
\end{equation}
Linear maps like $L_{s,k}$ are sometimes called positive transfer
operators or Perron-Frobenius operators and arise in many contexts
other than computation of Hausdorff dimension: see, for example,
\cite{Baladi}. If $r(L_{s,k})$ denotes the spectral radius of
$L_{s,k}$, then $\lambda_s = r(L_{s,k})$ is positive and independent of
$k$ for $0 \le k \le N$; and $\lambda_s$ is an algebraically simple
eigenvalue of $L_{s,k}$ with a corresponding unique, normalized strictly
positive eigenfunction $v_s \in C^N(\bar H)$.  Furthermore, the map
$s \mapsto \lambda_s$ is continuous.   If $\sigma(L_{s,k}) \subset \C$ denotes
the spectrum of the complexification of $L_{s,k}$, $\sigma(L_{s,k})$ depends
on $k$, but for $1 \le k \le N$,
\begin{equation}
\label{intro1.3}
\sup\{|z|: z \in \sigma(L_{s,k})\setminus\{\lambda_s\}\} < \lambda_s.
\end{equation}
If $k=0$, the strict inequality in \eqref{intro1.3} may fail. A more precise
version of the above result is stated in Theorem~\ref{thm:1.1} of this paper
and Theorem~\ref{thm:1.1} is a special case of results in \cite{E}.  The
method of proof involves ideas from the theory of positive linear operators,
particularly generalizations of the Kre{\u\i}n-Rutman theorem to noncompact
linear operators; see \cite{Krein-Rutman}, \cite{Bonsall},
\cite{Schaefer-Wolff}, \cite{L}, and \cite{Mallet-Paret-Nussbaum}. We do not
use the thermodynamic formalism (see \cite{Ruelle}) and often our operators
cannot be studied in Banach spaces of analytic functions.

The linear operators which are relevant for the computation of Hausdorff
dimension comprise a small subset of the transfer operators described in
\eqref{intro1.2}, but the analysis problem which we shall consider here can be
described in the generality of \eqref{intro1.2} and is of interest in this
more general context. We want to find rigorous methods to estimate
$r(L_{s,k})$ accurately and then use these methods to estimate $s_*$, where,
in our applications, $s_*$ will be the unique number $s \ge 0$ such that
$r(L_{s,k})=1$.  Under further assumptions, we shall see that $s_*$ equals
$\dim_H(C)$, the Hausdorff dimension of the invariant set associated to the
IFS.  This observation about Hausdorff dimension has been made, in varying
degrees of generality by many authors. See, for example, \cite{Bumby1},
\cite{Bumby2}, \cite{Bowen}, \cite{Cusick1}, \cite{Cusick2}, \cite{Falconer},
\cite{Good}, \cite{Hensley1}, \cite{Hensley2}, \cite{Hensley3},
\cite{J}, \cite{Jenkinson}, \cite{Jenkinson-Pollicott},
\cite{MR1902887}, \cite{H}, \cite{Mauldin-Urbanski}, \cite{N-P-L},
\cite{Ruelle}, \cite{Ruelle2}, \cite{Rugh}, and \cite{Schief}.

In the applications in this paper, we shall assume, for simplicity, that
$H$ is a bounded open interval, that $\theta_j: \bar
H \to \bar H$ is a $C^N$ contraction mapping, where $N \ge 3$, (or
more generally satisfies (H5.1)) and $\theta_j^{\prime}(x) \neq 0$ for
all $x \in \bar H$. In the notation of \eqref{intro1.2}, we define
$g_j(x) = |\theta_j^{\prime}(x)|$.  It is often natural to assume that $H$ is
a finite union of open intervals, and our methods apply with no essential
change to this case.

Given the existence of a strictly positive $C^N$ eigenfunction $v_s$
for \eqref{intro1.2}, we show in Section~\ref{sec:1d-deriv} for $1 \le
p \le 3$, that one can obtain explicit upper and lower bounds for the
quantity $D^p v_s(x)/v_s(x)$ for $x \in \bar H$, where $D^p$ denotes
the $p$-th derivative of $v_s$. Such bounds can also be obtained for
$p>3$, but calculations become more onerous.  In the important special
case that $\theta_j(x)$ is of the form $(x+b_j)^{-1}$, where $b_j >0$ and
$g_j(x) = |\theta_j^{\prime}(x)|$, we obtain in
Section~\ref{sec:mobius} sharp estimates on the quantity $D^p
v_s(x)/v_s(x)$ for all $p \ge 1$ and all $x \in \bar H$. These
estimates play a crucial role in allowing us to obtain rigorous upper
and lower bounds for the Hausdorff dimension.

The basic idea of our numerical scheme is to cover $\bar H$ by nonoverlapping
intervals of length $h$.  We remark that our collection of intervals need not
be a {\it Markov partition} for our IFS; compare the use of {\it Markov
partitions} in \cite{McMullen}.  We then approximate the strictly positive,
$C^2$ eigenfunction $v_s$ by a continuous piecewise linear function.  Using
explicit bounds on the first and second derivatives of $v_s$, we are able to
associate to the operator $L_{s,k}$, square matrices $A_s$ and $B_s$, which
have nonnegative entries and also have the property that $r(A_s) \le \lambda_s
\le r(B_s)$. We note that using a piecewise linear approximation to $v_s$, as
opposed to a piecewise constant approximation, leads to a considerable
increase in accuracy and speed of convergence. A key role here is played by an
elementary fact which is not as well known as it should be.  If $M$ is a
nonnegative matrix and $\wrm$ is a strictly positive vector and $M \wrm \le
\lambda \wrm$, (coordinate-wise), then $r(M) \le \lambda$.  An analogous
statement is true if $M \wrm \ge \lambda \wrm$. We emphasize that our approach
is robust and allows us to study the case $H \subset \R$ when
$\theta_j(\cdot)$, $1 \le j \le m$, is only $C^3$.

If $s_*$ denotes the unique value of $s$ such that $r(L_{s_*}) =
\lambda_{s_*} = 1$, so that $s_*$ is the Hausdorff dimension of the
invariant set for the IFS under study, we proceed as follows.  If we
can find a number $s_1$ such that $r(B_{s_1}) \le 1$, then, since the
map $s \mapsto \lambda_s$ is decreasing, $\lambda_{s_1} \le r(B_{s_1})
\le 1$, and we can conclude that $s_* \le s_1$.  Analogously, if we
can find a number $s_2$ such that $r(A_{s_2}) \ge 1$, then
$\lambda_{s_2} \ge r(A_{s_2}) \ge 1$, and we can conclude that $s_*
\ge s_2$.  By choosing the mesh size for our approximating piecewise
polynomials to be sufficiently small, we can make $s_1-s_2$ small,
providing a good estimate for $s_*$.  For a given $s$, $r(A_s)$ and
$r(B_s)$ are easily found by variants of the power method for
eigenvalues, since (see Section~\ref{sec:compute-sr}) the largest eigenvalue has
multiplicity one and is the only eigenvalue of its modulus.

If the coefficients $g_j(\cdot)$ and the maps $\theta_j(\cdot)$ in
\eqref{intro1.2} are $C^N$ with $N >2$, it is natural to approximate
$v_s(\cdot)$ by piecewise polynomials of degree $N-1$.
The corresponding matrices $A_s$ and $B_s$ may no longer have all
nonnegative entries and the arguments of this paper are no longer
directly applicable.  However, we hope to prove in a future paper
that inequalities like $r(A_s) \le \lambda_s \le r(B_s)$ remain true and
lead to much improved upper and lower bounds for $r(L_s)$.  Heuristic
evidence for this assertion is given in Table~\ref{tb:t2} of 
Section~\ref{subsec:cantor-num}.

We illustrate our new approach by first considering in 
Section~\ref{sec:1dexps} the computation of the Hausdorff dimension
of invariant sets in $[0,1]$ arising from classical continued fraction
expansions.  In this much studied case, one defines $\theta_m =
1/(x+m)$, for $m$ a positive integer and $x \in [0,1]$; and for a
subset $\B \subset \N$, one considers the IFS $\{\theta_m \, | \, m
\in \B\}$ and seeks estimates on the Hausdorff dimension of the
invariant set $C =C(\B)$ for this IFS.  This problem has previously
been considered by many authors. See \cite{Bourgain-Kontorovich},
\cite{Bumby1}, \cite{Bumby2}, \cite{Good}, \cite{Hensley1},
\cite{Hensley2}, \cite{Hensley3}, \cite{Jenkinson},
\cite{Jenkinson-Pollicott}, and \cite{Heinemann-Urbanski}.  In this case,
\eqref{intro1.2} becomes
\begin{equation*}
(L_{s,k}w)(x) = \sum_{m \in \B} \Big(\frac{1}{x+m}\Big)^{2s} 
w\Big(\frac{1}{x+m}\Big), \qquad 0 \le x \le 1,
\end{equation*}
and one seeks a value $s \ge 0$ for which $\lambda_s:= r(L_{s,k}) =1$.
Table~\ref{tb:t1} in Section~\ref{subsec:cantor-num} gives upper and lower
bounds for the value $s$ such that $\lambda_s =1$ for various sets
$\B$. Jenkinson and Pollicott \cite{Jenkinson-Pollicott} use a completely
different method and obtain, when $|\B|$ is small, high accuracy estimates for
$\dim_H(C(\B))$, in which successive approximations converge at a
super-exponential rate.  It is less clear (see \cite{Jenkinson})
how well the approximation scheme in \cite{Jenkinson-Pollicott} or
\cite{Jenkinson} works when $|\B|$ is moderately large or when different real
analytic functions $\hat \theta_j: [0,1] \to [0,1]$ are used.  Here, in the
one dimensional case, we present an alternative approach with much wider
applicability that only requires the maps in the IFS to be $C^3$.  As an
illustration, we consider in Section~\ref{subsec:lowreg} perturbations of the
IFS for the middle thirds Cantor set for which the corresponding contraction
maps are $C^3$, but not $C^4$.

It is also worth comparing the approach used in our paper with that
of McMullen \cite{McMullen}. Superficially the methods seem different,
but there are underlying connections.  We exploit the existence of a
$C^k$, strictly positive eigenfunction $v_s$ of \eqref{intro1.2} with
eigenvalue $\lambda_s$ equal to the spectral radius of $L_{s,k}$; and we
observe that explicit bounds on derivatives of $v_s$ can be
exploited to prove convergence rates on numerical approximation schemes
which approximate $\lambda_s$.  McMullen does not explicitly mention
the operator $L_{s,k}$ or the analogue of $L_{s,k}$ for graph directed
iterated function systems, and he does not use $C^k$, strictly positive
eigenfunctions of equations like \eqref{intro1.2}.  Instead, he exploits
finite positive measures $\mu$ which are called {\it $\mathcal{F}-$invariant
densities of dimension $\delta$.} If $s_*$ is a value of $s$ for which
the above eigenvalue $\lambda_s =1$, then in our context the measure $\mu$
is an eigenfunction of the Banach space adjoint $(L_{s_*,0})^*$ with
eigenvalue $1$, and our $s_*$ corresponds to $\delta$ above.  Standard
arguments using weak$^*$ compactness, the Schauder-Tychonoff fixed
point theorem, and the Riesz representation theorem imply the existence
of a regular, finite, positive, complete measure $\mu$, defined on a
$\sigma$-algebra containing all Borel subsets of the underlying space
$\bar H$ and such that $(L_{s_*,0})^* \mu = \mu$ and $\int v_s \, d \mu =1$.

McMullen also uses refinements of {\it Markov partitions}, while our
partitions, both here and in a sequel \cite{hdcomp2} in which we consider
two dimensional problems, need not be Markov. However, in the end, both
approaches generate (different) $n \times n$ nonnegative matrices $M_s$,
parametrized by a parameter $s$ and both methods use the spectral radius of
$M_s$ to approximate the desired Hausdorff dimension $s_*$.  McMullen's
matrices are obtained by approximating certain nonconstant functions defined
on a refinement of the original Markov partition by piecewise constant
functions defined with respect to this refinement.  We approximate by
linear functions on each subset in our partition in dimension one and
(see \cite{hdcomp2}) by bilinear functions defined on each subset
of our partition in dimension two.  As we show below, by exploiting
estimates on higher derivatives of $v_s(\cdot)$, our methods give
explicit upper and lower bounds for $s_*$ and more rapid convergence
to $s_*$ than one obtains using piecewise constant approximations.

The square matrices $A_s$ and $B_s$ mentioned above and described in more
detail in Section~\ref{sec:1dexps} have nonnegative entries and satisfy
$r(A_s) \le \lambda_s \le r(B_s)$.  To apply standard numerical methods, it is
useful to know that all eigenvalues $\mu \neq r(A_s)$ of $A_s$ satisfy $|\mu|
< r(A_s)$ and that $r(A_s)$ has algebraic multiplicity one and that
corresponding results hold for $r(B_s)$.  Such results are proved in
Section~\ref{sec:compute-sr} when the mesh size, $h$, is sufficiently small.
Note that this result does not follow from the standard theory of nonnegative
matrices, since $A_s$ and $B_s$ typically have zero columns and are not
primitive. We also prove that $r(A_s) \le r(B_s) \le (1 + C_1 h^2) r(A_s)$,
where the constant $C_1$ can be explicitly estimated.  In
Section~\ref{sec:logconvex}, we prove that the map $s \mapsto \lambda_s$ is
log convex and strictly decreasing; and the same result is proved for $s
\mapsto r(M_s)$, where $M_s$ is a naturally defined matrix such that $A_s \le
M_s \le B_s$.

In a subsequent paper \cite{hdcomp2}, we consider the computation of the
Hausdorff dimension of some invariant sets arising for complex continued
fractions.  Suppose that $\B$ is a subset of $I_1 = \{m+ni \, | \, m \in \N,
n\in \Z\}$, and for each $\be \in \B$, define $\theta_{\be}(z) =
(z+\be)^{-1}$. Note that $\theta_{\be}$ maps $\bar G = \{z \in \C \, | \,
|z-1/2| \le 1/2\}$ into itself.  We are interested in the Hausdorff dimension
of the invariant set $C = C(\B)$ for the IFS $\{\theta_b \, |\, \be \in \B\}$.
This is a two dimensional problem and we allow the possibility that $\B$ is
infinite. In general (contrast work in \cite{Jenkinson-Pollicott} and
\cite{Jenkinson}), it does not seem possible in this case to replace
$L_{s,k}$, $k \ge 2$, by an operator $\Lambda_s$ acting on a Banach space of
analytic functions of one complex variable and satisfying $r(\Lambda_s) =
r(L_{s,k})$.  Instead, we work in $C^2(\bar G)$ and apply our methods to
obtain rigorous upper and lower bounds for the Hausdorff dimension
$\dim_H(C(\B))$ for several examples. The case $\B = I_1$ has been of
particular interest and is one motivation for the paper \cite{hdcomp2}.  In
\cite{Gardner-Mauldin}, Gardner and Mauldin proved that $d:= \dim_H(C(I_1))
<2$. In Theorem 6.6 of \cite{Mauldin-Urbanski}, Mauldin and Urbanski proved
that $1.2484 <d \le 1.885$, and in \cite{Priyadarshi}, Priyadarshi proved that
$d \ge 1.78$.  We prove that $1.85550 \le d \le 1.85589$.  A combination
of the results in this paper plus the subsequent paper \cite{hdcomp2}
can be found in a preliminary version published on the arXiv \cite{hdcomplt}.

Although many of the key results in the paper are described above, an outline
summarizing the sections may be helpful. In Section~\ref{sec:prelim}, we
recall the definition of Hausdorff dimension and present some mathematical
preliminaries. In Section~\ref{sec:1dexps}, we present the details of our
approximation scheme for Hausdorff dimension, explain the crucial role played
by estimates on derivatives of order $\le 2$ of $v_s$, and give the
aforementioned estimates for Hausdorff dimension.  We emphasize that this is a
feasibility study. We have limited the accuracy of our approximations to what
is easily found using the standard precision of {\it Matlab} and have run only
a limited number of examples, using mesh sizes that allow the programs to run
fairly quickly. In addition, we have not attempted to exploit the special
features of our problems, such as the fact that our matrices are sparse.
Thus, it is clear that one could write a more efficient code that would also
speed up the computations.  However, the {\it Matlab} programs we have
developed are available on the web at {\tt
  www.math.rutgers.edu/\char'176falk/hausdorff/codes.html}, and we hope other
researchers will run other examples of interest to them.

The theory underlying the work in Section~\ref{sec:1dexps} is deferred to
Sections~\ref{sec:exist}--\ref{sec:logconvex}.  In Section~\ref{sec:exist} we
describe some results concerning existence of $C^m$ positive eigenfunctions
for a class of positive (in the sense of order-preserving) linear operators.
In Section~\ref{sec:1d-deriv}, we derive explicit bounds on the derivatives of
the eigenfunction $v_s$ of $L_s$ and in  Section~\ref{sec:mobius}, we show how
much sharper bounds on the derivatives of the eigenfunction can be obtained
when the maps $\theta_b$ are M\"obius transformations.
In Section~\ref{sec:compute-sr}, we verify some spectral properties of the
approximating matrices which justify standard numerical algorithms for
computing their spectral radii. Finally, in Section~\ref{sec:logconvex}, we
show the log convexity of the spectral radius $r(L_s)$, which we exploit in
our numerical approximation scheme.

\section{Preliminaries}
\label{sec:prelim}
We recall the definition of the Hausdorff dimension, $\dim_H(K)$, of
a subset $K \subset \R^N$.  To do so, we first define for a given $s \ge0$
and each set $K \subset \R^N$, 
\begin{equation*}
H_{\delta}^s(K) = \inf\{\sum_i |U_i|^s: \{U_i\} \text{ is a } \delta \text{ cover
of } K\},
\end{equation*}
where $|U|$ denotes the diameter of $U$ and a countable collection $\{U_i\}$
of subsets of $\R^N$ is a $\delta$-cover of $K \subset \R^N$ if
$K \subset \cup_i U_i$ and $0 < |U_i| < \delta$ for all $i$.  We then define
the $s$-dimensional Hausdorff measure
\begin{equation*}
H^s(K) = \lim_{\delta \rightarrow 0+} H_{\delta}^s(K).
\end{equation*}
Finally, we define the Hausdorff dimension of $K$, $\dim_H(K)$, as
\begin{equation*}
\dim_H(K) = \inf\{s: H^s(K) =0\}.
\end{equation*}

We now state the main result connecting Hausdorff dimension to the spectral
radius of the map defined by \eqref{intro1.2}.  To do so, we first define the
concept of an {\it infinitesimal similitude} (sometimes called a conformal
map).  Let $(S,d)$ be a perfect metric space. If $\theta:S \to S$,
then $\theta$ is an infinitesimal similitude at $t \in S$ if for any sequences
$(s_k)_k$ and $(t_k)_k$ with $s_k \neq t_k$ for $k \ge 1$ and $s_k \rightarrow
t$, $t_k \rightarrow t$, the limit
\begin{equation*}
\lim_{k \rightarrow \infty} \frac{d(\theta(s_k), \theta(t_k)}{d(s_k,t_k)}
=: (D \theta)(t)
\end{equation*}
exists and is independent of the particular sequences 
$(s_k)_k$ and $(t_k)_k$.  Furthermore, $\theta$ is an infinitesimal similitude
on $S$ if $\theta$ is an infinitesimal similitude at $t$ for all $t \in S$.

This concept generalizes the concept of affine linear similitudes, which
are affine linear contraction maps $\theta$ satisfying
for all $x,y \in \R^n$
\begin{equation*}
d(\theta(x), \theta(y)) = c d(x,y), \quad c \neq 0.
\end{equation*}
In particular, the examples discussed in this paper, such as maps
of the form $\theta(x) = 1/(x+m)$, with $m$ a positive integer,
are infinitesimal similitudes. More generally, if $S$ is a compact subset
of $\R^1$ and $\theta:S \to S$ extends to a $C^1$ map defined on an
open neighborhood of $S$ in $\R^1$, then $\theta$ is an infinitesimal
similitude.

\begin{thm} (Theorem 1.2 of \cite{N-P-L}.)
Let $\theta_i:S \to S$ for $1 \le i \le N$ be infinitesimal similitudes
and assume that the map $t \mapsto (D\theta_i)(t)$ is a strictly positive
H\"older continuous function on $S$.  Assume that $\theta_i$ is a Lipschitz
map with Lipschitz constant $c_i \le c <1$ and let
$C$ denote the unique, compact, nonempty invariant set such that
\begin{equation*}
C = \cup_{i=1}^N \theta_i(C).
\end{equation*}
Further, assume that $\theta_i$ satisfy
\begin{equation*}
\theta_i(C) \cap \theta_j(C) = \emptyset, \text{ for } 1 \le i,j \le N.
\ i \neq j
\end{equation*}
and are one-to-one on $C$.  Then the Hausdorff dimension of $C$ is
given by the unique $\sigma_0$ such that $r(L_{\sigma_0}) = 1$.
\end{thm}
For related results on the computation of Hausdorff dimension, we refer
the reader to the list of references near the bottom of page 2.

Finally, we state a result that is key to obtaining explicit upper and lower
bounds on the Hausdorff dimension. Although we give a proof to keep our
presentation self-contained, the following lemma is actually a special case of
much more general results concerning order-preserving, homogeneous cone
mappings: see Lemmas 9.1-9.4 on pages 89-91 in \cite{Y} and also Lemma 2.2 in
\cite{C} and Theorem 2.2 in \cite{B}.  If, for $\wrm$ as in
Lemma~\ref{lem:nonneg} below, we let $D$ denote the positive diagonal $N
\times N$ matrix with diagonal entries $w_j$, $1 \le j \le N$, $r(M) =
r(D^{-1}MD)$; and Lemma~\ref{lem:nonneg} can also be obtained by applying
Theorem 1.1 on page 24 of \cite{D} to $D^{-1}MD$.
\begin{lem}
\label{lem:nonneg}
Let $M$ be an $N \times N$ matrix with non-negative
entries and $\wrm$ an $N$ vector with strictly positive components.
\begin{align*}
\text{If  } (M\wrm)_k &\ge \lambda \wrm_k, \quad k =1, \ldots N, \qquad
\text{then }  r(M) \ge \lambda,
\\
\text{If  } (M\wrm)_k &\le \lambda \wrm_k, \quad k =1, \ldots N, \qquad
\text{then }   r(M) \le \lambda.
\end{align*}
\end{lem}
\begin{proof}
If $(M\wrm)_k \ge \lambda \wrm_k$, $k =1, \ldots N$, it easily follows that
 $(M^n\wrm)_k \ge \lambda^n \wrm_k$ and so $\|M^n \wrm\|_{\infty} \ge \lambda^n
\|\wrm\|_{\infty}$.  Let $\erm$ be vector with all $\erm_i=1$. Then
\begin{equation*}
\|M^n\|_{\infty} = \|M^n \erm\|_{\infty} \ge \|M^n \wrm\|_{\infty}/\|\wrm\|_{\infty}
\ge \lambda^n.
\end{equation*}
Hence,
\begin{equation*}
r(M) = \lim_{n \rightarrow \infty} \|M^n\|_{\infty}^{1/n} \ge \lambda.
\end{equation*}

If $(M\wrm)_k \le \lambda \wrm_k$, $k =1, \ldots N$, it easily follows that
 $(M^n\wrm)_k \le \lambda^n \wrm_k$.  Let $k$ be chosen so that
$\|M^n\|_{\infty} = \sum_j (M^n)_{k,j}$. 
Since $[r(M)]^n = r(M^n) \le \|M^n\|_{\infty}$,
\begin{equation*}
\min_j \wrm_j [r(M)]^n \le \min_j \wrm_j \sum_j (M^n)_{k,j}
\le \sum_j (M^n)_{k,j} \wrm_j = (M^n\wrm)_k \le \lambda^n \wrm_k.
\end{equation*}
So,
\begin{equation*}
\min_j \wrm_j \le [\lambda/r(M)]^n \wrm_k.
\end{equation*}
If $r(M) > \lambda$, then letting $n \rightarrow \infty$, we get that
$\min_j \wrm_j \le 0$, which contradicts the fact that all $\wrm_j > 0$.
Hence, $r(M) \le \lambda$.
\end{proof}

\section{Examples}
\label{sec:1dexps}
\subsection{Continued fraction Cantor sets}
\label{subsec:cantor}
We first consider the problem of computing the Hausdorff dimension of some
Cantor sets arising from continued fraction expansions.  More precisely, given
any number $0<x<1$, we can consider its continued fraction expansion
\begin{equation*}
x = [a_1,a_2,a_3, \ldots] = 
\cfrac{1}{a_1 + \cfrac{1}{a_2 + \cfrac{1}{a_3 + \cdots}}},
\end{equation*}
where $a_1,a_2,a_3, \ldots \in \N$.  We then consider the Cantor set $E_{[m_1,
  \ldots, m_p]}$, of all points in $[0,1]$ where we restrict the coefficients
$a_i$ to the values $m_1, \ldots, m_p$. A number of papers (e.g.,
\cite{Bumby1}, \cite{Bumby2}, \cite{Good}, \cite{Hensley1}, \cite{Hensley2},
\cite{Jenkinson-Pollicott}) have considered this problem in the case of the
set $E_{1,2}$, consisting of all points in $[0,1]$ for which each $a_i$ has
the value $1$ or $2$.  In \cite{Jenkinson-Pollicott}, a method is presented
that computes this dimension to 25 decimal places.  Computations are also
presented in that paper and in \cite{Jenkinson} for other choices of the
values $m_1, \ldots, m_p$. In \cite{Bourgain-Kontorovich}, the Hausdorff
dimension of the Cantor set $E_{2,4,6,8,10}$ is computed to three decimal
places (0.517).

Corresponding to the choices of $m_i$, we associate contraction maps
$\theta_m(x) = 1/(x+m)$.  A key fact is that the Cantor sets we consider
can be generated as limit points of sequences of these contraction maps.  For
example, the set $E_{1.2}$ can be generated using the maps
$\theta_1(x) = 1/(x+1)$ and $\theta_2(x) = 1/(x+2)$ as the set of limit points
of sequences $\theta_{m_1} \ldots \theta_{m_n}(0)$, for $m_1, m_2, \ldots
\in \{1,2\}$.

For $w \in C[0,1]$, we define
\begin{equation*}
(L_sw)(x) = \sum_{j=1}^p \Big|\theta^{\prime}_{m_j}(x)\Big|^s w(\theta_{m_j}(x)).
\end{equation*}
In fact, we can just as easily think of $L_s$ as an operator on $C[0,
\gamma^{-1}]$ or $C[(1 + \Gamma)^{-1}, \gamma^{-1}]$, where $\gamma = \min
m_j$ and $\Gamma = \max m_j$.  In the discussion below, we will usually
work on the interval $[0, \gamma^{-1}]$.

Our computations are based on the following result, which we shall prove
in subsequent sections.
\begin{thm}
\label{thm:posev1d}
For all $s >0$, $L_s$ has a unique strictly positive eigenfunction $v_s$ with
$L_s v_s = \lambda_s v_s$, where $\lambda_s >0$ and $\lambda_s = r(L_s)$, the
spectral radius of $L_s$.  Furthermore, the map $s \mapsto \lambda_s$ is
strictly decreasing and continuous, and for all $p >0$ and for all 
$x \in [0,\gamma^{-1}]$,
\begin{multline}
\label{prop1}
(2s)(2s+1) \cdots (2s+p-1)(2 \gamma^{-1} + \Gamma)^{-p} \le
(-1)^p \frac{D^p [v_s(x)] }{v_s(x)}
\\
\le (2s)(2s+1) \cdots (2s+p-1)\gamma^{-p},
\end{multline}
where $\gamma = \min_j m_j$ and $\Gamma = \max_j m_j$. 
Finally, the Hausdorff dimension of the Cantor set generated from the maps

$\theta_{m_1}$, $\ldots$, $\theta_{m_p}$ is the unique value of $s$ with
$\lambda_s=1$.
\end{thm}
Note that it follows easily from \eqref{prop1} when $p=1$ and $x_1, x_2 \in
[0,1]$ , that
\begin{equation}
\label{prop2}
v_s(x_2) \le v_s(x_1) \exp(2s|x_2-x_1|/\gamma).
\end{equation}
To see this, write
\begin{equation*}
\log\frac{v_s(x_2)}{v_s(x_1)} = \log v_s(x_2) - \log v_s(x_1) =
\int_{x_1}^{x_2} \frac{d}{dx} \log v_s(x) \, dx = \int_{x_1}^{x_2}
\frac{v_s^{\prime}(x)}{v_s(x)} \, dx,
\end{equation*}
apply the bound in \eqref{prop1}, and exponentiate the result.

To obtain approximations of the dimension of the Cantor sets described
in this section, we first approximate a function $f \in C^2[0,\gamma^{-1}]$ by
a continuous, piecewise linear function defined on a mesh of interval size
$h$ on $[0,\gamma^{-1}]$. More specifically, we approximate
$f(x)$, $x_k \le x \le x_{k+1}$ by its piecewise linear interpolant $f^I(x)$
given by
\begin{equation*}
f^I(x) = \frac{x_{k+1}-x}{h} f(x_k) + \frac{x-x_k}{h} f(x_{k+1}),
\quad x_k \le x \le x_{k+1},
\end{equation*}
where the mesh points $x_k$ satisfy $0 = x_0 < x_1 , \dots < x_n =\gamma^{-1}$,
with $x_{k+1} - x_k = h = 1/(\gamma n)$. 

Notice that if $\wrm = (\wrm_0, \ldots,\wrm_n)$ is a vector in $\R^{n+1}$, we
can associate a continuous piecewise linear function $w^I: [0,\gamma^{-1}] \to
\R$ defined with respect to the partition $0 = x_0 < x_1 < \ldots <x_n
=\gamma^{-1}$ of $[0,\gamma^{-1}]$ by:
\begin{equation*}
w^I(y) =
\frac{[x_{r+1} - y]}{h} (\wrm)_{r}
+ \frac{[y-x_{r}]}{h} (\wrm)_{r +1}, \qquad y \in [x_r,x_{r+1}], \quad 0 \le r
< n.
\end{equation*}
This notation will be used below and will play an important role in our
argument.

Our goal is to construct $(n+1) \times (n+1)$ matrices $A_s$ and $B_s$ which
have nonnegative entries and satisfy
\begin{equation*}
r(A_s) \le r(L_s) \le r(B_s),
\end{equation*}
where $r(A_s)$ (respectively, $r(B_s)$) denotes the spectral radius of $A_s$
(respectively, $B_s$).  Furthermore, the entries $(A_s)_{ij}$ and $(B_s)_{ij}$
of $A_s$ and $B_s$ satisfy (for $n$ large)
\begin{equation*}
0 \le (B_s)_{ij} - (A_s)_{ij} \le Ch^2,
\end{equation*}
where $C$ is a constant which can be estimated explicitly and is independent
of $n$.

Standard results for the error in linear interpolation on an interval $[a,b]$
(e.g., see Theorem 3.2 of \cite{Atkinson}) assert that for $x \in [a,b]$,
there exists $\xi = \xi(x) \in (a,b)$ such that
\begin{equation*}
f^I(x) - f(x):= \frac{b-x}{b-a}f(b) 
+ \frac{x-a}{b-a} f(a) - f(x)  = \frac{1}{2}(b - x)(x-a) f^{\prime\prime}(\xi).
\end{equation*}
In the notation above, if $x \in [0,\gamma^{-1}]$ and $x_r \le x \le x_{r+1}$ for
some $r$, $0 \le r <n$, we shall apply this error estimate with $a= x_r$
and $b = x_{r+1}$, so $\xi \in (x_r,x_{r+1})$.

We can also use results from Theorem~\ref{thm:posev1d} to bound the
interpolation error.
Letting $f(x) = v_s(x)$, we obtain from Theorem~\ref{thm:posev1d} that
\begin{equation*}
2s (2s+1)(2 \gamma^{-1} + \Gamma)^{-2} v_s(\xi)
 \le v_s^{\prime\prime}(\xi) \le 2s (2s+1)\gamma^{-2} v_s(\xi).
\end{equation*}
Using \eqref{prop2}, and the fact that $|\xi- x_r| \le h$ for
$\xi \in [x_r,x_{r+1}]$, we have
\begin{multline*}
v_s(x_{r}) \exp(- 2s h/\gamma) \le v_s(x_{r}) \exp(-2s|\xi -
x_{r}|/\gamma) \le v_s(\xi)
\\
\le v_s(x_{r}) \exp(2s|\xi - x_{r}|/\gamma) \le v_s(x_{r}) \exp(2s h/\gamma).
\end{multline*}
Similarly,
\begin{equation*}
v_s(x_{r+1}) \exp(- 2s h/\gamma) \le v_s(\xi)
\le v_s(x_{r+1}) \exp(2s h/\gamma).
\end{equation*}
Taking a suitable convex combination of these results, we get
for $y \in [x_r, x_{r+1}]$,
\begin{equation*}
v_s^I(y) \exp(- 2s h/\gamma) \le v_s(\xi)
\le v_s^I(y) \exp(2s h/\gamma).
\end{equation*}
Using the interpolation error estimate, we then get
for $x_{r} \le y \le x_{r+1}$,
\begin{multline*}
[x_{r+1} - y][y -x_r]
s (2s+1) (2\gamma^{-1} + \Gamma)^{-2} \exp(-2sh/\gamma) \, v_s^I(y)
\le  v_s^I(y) - v_s(y) 
\\
\le 
[x_{r+1} - y][y -x_r]
s (2s+1) \gamma^{-2} \exp(2sh/\gamma) \, v_s^I(y).
\end{multline*}

Using this estimate, we have precise upper and lower bounds on the error
in the interval $[x_r,x_{r+1}]$ that only depend on the function values of
$v_s$ at $x_{r}$ and $x_{r+1}$.  For $y \in [x_r, x_{r+1}]$, define
error functionals
\begin{align*}
\err^1(y) &= [x_{r+1} - y][y -x_{r}]
s (2s+1) \gamma^{-2} \exp(2sh/\gamma),
\\
\err^2(y) &= [x_{r+1} - y][y -x_{r}]
s (2s+1) (2\gamma^{-1} + \Gamma)^{-2} \exp(-2sh/\gamma).
\end{align*}
Note that $\err^1(y)$ and $\err^2(y)$ depend on the subinterval in which
$y$ lies, although this is not reflected directly in the notation.

It then follows that for all $y \in [x_r, x_{r+1}]$,
\begin{equation*}
[1 - \err^1(y)] v_s^I(y) \le v_s(y) 
\le [1 - \err^2(y)] v_s^I(y).
\end{equation*}
For a fixed $k$, $0 \le k \le n$, if we replace $y$ in the above inequality
by $\theta_{m_j}(x_k)$ and sum over $j$, we obtain
\begin{multline*}
\sum_{j=1}^p \Big|\theta^{\prime}_{m_j}(x_k)\Big|^s 
[1 - \err^1(\theta_{m_j}(x_k))] v_s^I(\theta_{m_j}(x_k))
\le \sum_{j=1}^p \Big|\theta^{\prime}_{m_j}(x_k)\Big|^s v_s(\theta_{m_j}(x_k))
\\
= (L_s v_s)(x_k) =  r(L_s) v_s(x_k)
\le \sum_{j=1}^p \Big|\theta^{\prime}_{m_j}(x_k)\Big|^s 
[1 - \err^2(\theta_{m_j}(x_k))] v_s^I(\theta_{m_j}(x_k)).
\end{multline*}

Motivated by the above inequality, we now define $(n+1) \times (n+1)$ matrices
$A_s$ and $B_s$ which have nonnegative entries and satisfy the property that
$r(A_s) \le r(L_s) \le r(B_s)$. Letting $\wrm$ be a vector in $\R^{n+1}$, we
define $(B_s \wrm)_k$ and $(A_s \wrm)_k$, the $k$th component of $B_s \wrm$
and $A_s \wrm$ respectively,  by
\begin{align*}
(B_s \wrm)_k &= \sum_{j=1}^p \Big|\theta^{\prime}_{m_j}(x_k)\Big|^s 
[1 - \err^2(\theta_{m_j}(x_k))] w^I(\theta_{m_j}(x_k)),
\\
(A_s \wrm)_k &= \sum_{j=1}^p \Big|\theta^{\prime}_{m_j}(x_k)\Big|^s 
[1 - \err^1(\theta_{m_j}(x_k))] w^I(\theta_{m_j}(x_k)).
\end{align*}
Because of the fact that in all of our previous definitions, we take $0 \le k
\le n$, we shall also do so in our definitions of $A_s$ and $B_s$, so that
these matrices have row and columns, numbered $0$ through $n$.  In the above
definitions, if $\theta_{m_j}(x_k) \in [x_{r_j}, x_{r_j +1}]$, (the
subinterval also depends on $k$, but we have omitted this dependence in the
notation, thinking of $k$ as fixed), then applying the previous definition of
$w^I(y)$,
\begin{equation*}
w^I(\theta_{m_j}(x_k)) = \frac{x_{r_j +1} - \theta_{m_j}(x_k)}{h} \wrm_{r_j}
+ \frac{\theta_{m_j}(x_k) - x_{r_j}}{h} \wrm_{r_j+1}.
\end{equation*}

To understand these formulas, note that $w^I(\theta_{m_j}(x_k))$ is just a
linear combination of two components of the vector $\wrm$, namely $\wrm_{r_j}$
and $\wrm_{r_j+1}$, where $x_{r_j}$ and $x_{r_j+1}$ are the endpoints of the
subinterval to which $\theta_{m_j}(x_k)$ belongs. Determining this subinterval
for $1 \le j \le p$ and $0 \le k \le n$ are the first calculations we need to
make.  In the case $p=1$, there is only one term in the sum (when $j=1$), and
since $(B_s \wrm)_k = \sum_{i=0}^n (B_s)_{k,i} \wrm_i$, we then have
\begin{align*}
(B_s)_{k,r_j} &= \Big|\theta^{\prime}_{m_j}(x_k)\Big|^s 
[1 - \err^2(\theta_{m_j}(x_k)] \frac{[x_{r_j+1} - \theta_{m_j}(x_k)]}{h},
\\
(B_s)_{k,r_j+1} &= \Big|\theta^{\prime}_{m_j}(x_k)\Big|^s 
[1 - \err^2(\theta_{m_j}(x_k)]\frac{[\theta_{m_j}(x_k)-x_{r_j}]}{h},
\\
(B_s)_{k,i} &= 0, \quad i \neq r_j, r_{j}+1.
\end{align*}
If $p >1$, then for each $j =2, \ldots, p$, we modify the entries in the $k$th
row of the matrix $B_s$, according to which subinterval the points
$\theta_{m_j}(x_k)$ lie. If the subinterval is disjoint from the previous
subintervals, then we need to modify the corresponding two columns of the
$k$th row of the matrix $B_s$, which introduces two new nonzero entries.  If
it coincides with a previous subinterval, then we simply add to the
coefficients in the two corresponding columns.  We perform this procedure
for each $x_k, k =0, \ldots, n$, thus generating the $n+1$ rows of the matrix
$B_s$.  The entries of the matrix $A_s$ are generated in a similar fashion.

An example, where we simplify the presentation by working on the interval
$[0,1]$ instead of $[0, \gamma^{-1}]$, is when $h=1/4$, so that we have $x_0
=0$, $x_1 =1/4$, $x_2 =1/2$, $x_3 = 3/4$, and $x_4 =1$. We only show the
computations for $B_s$, which is a $5 \times 5$ matrix, since the computations
for $A_s$ are similar.  If we consider $p=2$, $\theta_{m_1}(x) = 1/(x+3)$ and
$\theta_{m_2}(x) = 1/(x+5)$, then
\begin{gather*}
\theta_{m_1}(x_0) = \frac{1}{3},  \quad
\theta_{m_1}(x_1) = \frac{4}{13},\quad
\theta_{m_1}(x_2) = \frac{2}{7}, \quad
\theta_{m_1}(x_3) = \frac{4}{15}, \quad
\theta_{m_1}(x_4) = \frac{1}{4},
\\
\theta_{m_2}(x_0) = \frac{1}{5}, \quad
\theta_{m_2}(x_1) = \frac{4}{21}, \quad
\theta_{m_2}(x_2) = \frac{2}{11},\quad
\theta_{m_2}(x_3) = \frac{4}{23},\quad
\theta_{m_2}(x_4) = \frac{1}{6}.
\end{gather*}
Note that in this case, $\theta_{m_1}(x_k) \in [1/4,1/2]$ and
$\theta_{m_2}(x_k) \in [0,1/4]$, for $k=0, \ldots, 4$.  Although
$\theta_{m_1}(x_4)$ is also in $[1/4,1/2]$, there is no ambiguity, since the
only nonzero coefficient multiplies $\wrm_1$ and the coefficient is the same
with either choice of subinterval.

We next compute $w^I(\theta_{m_j}(x_k))$ and $\err^2(\theta_{m_j}(x_k))$.
\begin{align*}
w^I(\theta_{m_1}(x_k)) &= \frac{x_{2} - \theta_{m_1}(x_k)}{h} \wrm_{1}
+ \frac{\theta_{m_1}(x_k) - x_{1}}{h} \wrm_{2},
\\
w^I(\theta_{m_2}(x_k)) &= \frac{x_{1} - \theta_{m_2}(x_k)}{h} \wrm_{0}
+ \frac{\theta_{m_2}(x_k) - x_{0}}{h} \wrm_{1},
\\
\err^2(\theta_{m_1}(x_k)) &= [x_{2} - \theta_{m_1}(x_k)][\theta_{m_1}(x_k) - x_{1}]
s (2s+1) (2\gamma^{-1} + \Gamma)^{-2} \exp(-2sh/\gamma),
\\
\err^2(\theta_{m_2}(x_k)) &= [x_{1} - \theta_{m_2}(x_k)][\theta_{m_2}(x_k) - x_{0}]
s (2s+1) (2\gamma^{-1} + \Gamma)^{-2} \exp(-2sh/\gamma).
\end{align*}

Combining these results, we find that for $k =0, \ldots, 4$,
\begin{align*}
(B_s)_{k,0} &= \Big|\theta^{\prime}_{m_2}(x_k)\Big|^s 
[1 - \err^2(\theta_{m_2}(x_k)] [x_1 - \theta_{m_2}(x_k)]/h,
\\
(B_s)_{k,1} &= \Big|\theta^{\prime}_{m_1}(x_k)\Big|^s 
[1 - \err^2(\theta_{m_1}(x_k)] [x_2 - \theta_{m_1}(x_k)]/h
\\
&\qquad + \Big|\theta^{\prime}_{m_2}(x_k)\Big|^s 
[1 - \err^2(\theta_{m_2}(x_k)] [\theta_{m_2}(x_k) -x_0]/h,
\\
(B_s)_{k,2} &= \Big|\theta^{\prime}_{m_1}(x_k)\Big|^s 
[1 - \err^2(\theta_{m_1}(x_k)][\theta_{m_1}(x_k) -x_1]/h,
\\
(B_s)_{k,3} &= (B_s)_{k,4} =0.
\end{align*}

Returning to the general case, note that since $\err^i(y) = O(h^2)$ for
$i=1,2$, all of the entries of $A_s$ and $B_s$ will be nonnegative, provided
$h$ is sufficiently small.  However, the example given above is typical and
shows that, in general, the entries of $A_s$ and $B_s$ will not all be
strictly positive. If we define a vector $\wrm$ by $\wrm_k = v_s(x_k)$, then
$w^I(y) = v_s^I(y)$ for all $y \in [0,1]$, and our previous inequalities show
that for $0 \le k \le n$,
\begin{equation*}
(A_s \wrm)_k \le r(L_s) v_s(x_k) = r(L_s) \wrm_k, \qquad
(B_s \wrm)_k \ge r(L_s) v_s(x_k) = r(L_s) \wrm_k.
\end{equation*}
Since $\wrm_k = v_s(x_k) >0$ for $k=0, \ldots, n$, we can apply
Lemma~\ref{lem:nonneg} in Section~\ref{sec:prelim} about nonnegative matrices
to see that \begin{equation*} r(A_s) \le r(L_s) \le r(B_s).
\end{equation*}

As described in Section~\ref{sec:intro}, if $s_*$ denotes the unique value of
$s$ such that $r(L_{s_*}) = \lambda_{s_*} = 1$, then $s_*$ is the Hausdorff
dimension of the set $E_{[m_1, \ldots, m_p]}$.  If we can find a number $s_1$
such that $r(B_{s_1}) \le 1$, then $r(L_{s_1}) \le r(B_{s_1}) \le 1$, and we can
conclude that $s_* \le s_1$.  Analogously, if we can find a number $s_2$ such
that $r(A_{s_2}) \ge 1$, then $r(L_{s_2}) \ge r(A_{s_2}) \ge 1$, and we can
conclude that $s_* \ge s_2$.  By choosing the mesh sufficiently fine, we can
make $s_1-s_2$ small, providing a good estimate for $s_*$.

We can also reduce the number of computations by first iterating the maps
$\theta_{m_i}$ to produce a smaller initial domain that we need to
approximate.  For example, if we seek the Hausdorff dimension of the set
$E_{1,2}$, since $\theta_1([0,1]) = [1/2,1]$ and $\theta_2([0,1]) =
[1/3,1/2]$, the maps $\theta_1$ and $\theta_2$ map $[1/3,1] \mapsto [1/3,1]$,
so we can restrict the problem to this subinterval.  Further iterating, we see
that $\theta_1([1/3,1]) = [1/2,3/4]$ and $\theta_2([1/3,1]) = [1/3,3/7]$.
Hence the maps $\theta_1$ and $\theta_2$ map $[1/3,3/7] \cup [1/2,3/4]$ to
itself and we can further restrict the problem to this domain.

\subsection{Continued fraction Cantor sets -- numerical results}
\label{subsec:cantor-num}
In this section, we report in Table~\ref{tb:t1} the results of the
application of the algorithm described above to the computation of the
Hausdorff dimension of a sample of continued fraction Cantor sets.
Where the true value was known to sufficient accuracy, it is not hard
to check that the rate of convergence as $h$ is refined is $O(h^2)$,
which corresponds to the theoretical result described in
Remark~\ref{rem:9.9}. The upper and lower errors are computed based on
the results reported in \cite{Jenkinson-Pollicott}. For the last five
entries, we do not have independent results for the true solution
correct to a sufficient number of decimal places to compute the upper
and lower errors, but our results give an interval which must contain
the true solution.

Although the theory developed above does not
apply to higher order piecewise polynomial approximation, since one cannot
guarantee that the approximate matrices have nonnegative entries, we also
report in Table~\ref{tb:t2} and Table~\ref{tb:t22} the results of higher order
piecewise polynomial approximation to demonstrate the promise of this
approach.  In this case, we only provide the results for $B_s$, which does not
contain any corrections for the interpolation error.  In a future paper we
hope to prove that rigorous upper and lower bounds for the Hausdorff dimension
can also be obtained when higher order piecewise polynomial approximations are
used.

\begin{table}[ht]
\footnotesize
\caption{Computation of Hausdorff dimension $s$ of some continued fraction 
Cantor sets.}
\label{tb:t1}
\begin{center}
\begin{tabular}{|c|c|c|c|c|c|}
\hline
Set    &   h  &  lower $s$ & upper $s$ & low err & up err  \\
\hline \hline 
E[1,2] &  .0001 & 0.53128050509989  &  0.53128050644980
& 1.18e-09  & 1.73e-10  \\
       &   .00005 & 0.53128050598142 & 0.53128050632077
& 2.96e-10 &  4.36e-11  \\
\hline\hline  
E[1,3] & .0001 & 0.45448907685942 & 0.45448907780427 & 8.02e-10 & 1.42e-10
\\
&   .00005 & 0.45448907745903 &  0.45448907769761 & 2.03e-10 & 3.58e-11 
 \\
\hline \hline 
E[1,4]  & .0001 & 0.41118272409575 &  0.41118272491153 & 6.79e-10 & 1.37e-10 
 \\
&   .00005 & 0.41118272460331 &  0.41118272480924 & 1.71e-10  & 3.44e-11 
\\
\hline\hline 
E[2,3] & .0001 & 0.33743678074485 & 0.33743678082457 & 6.12e-11 & 1.85e-11
\\
&   .00005 & 0.33743678079023 & 0.33743678081090 & 1.58e-11& 4.84e-12
\\
\hline\hline 
E[2,4] & .0001 & 0.30631276799370 & 0.30631276807670 & 5.91e-11 & 2.39e-11
\\
&   .00005 & 0.30631276803924 & 0.30631276805816  & 1.35e-11 & 5.37e-12 
\\
\hline \hline 
E[10,11] & .0002 & 0.14692123539045 & 0.14692123539103 & 3.38e-13& 2.43e-13
\\
&   .00005 & 0.14692123539076 & 0.14692123539080 & 1.92e-14& 1.40e-14 
\\
\hline \hline 
E[100,10000] & .0004 & 0.05224659263866 & 0.05224659263866 & 2.21e-15 & 
3.50e-15 \\
&   .0001 & 0.05224659263866 & 0.05224659263866 & 1.73e-16 & 2.71e-16 
\\
\hline \hline 
E[2,4,6,8,10] & .0001 & 0.51735703083073 & 0.51735703098246 & & \\
&   .00005 & 0.51735703091123 & 0.51735703094801 & &  \\
\hline \hline 
E[1,\ldots,10] &  .0001 & 0.92573758921886 & 0.92573759153175 & &  \\
&   .00005 & 0.92573759066470 & 0.92573759124295 & &  \\
\hline\hline 
E[1,3, 5, \ldots, 33]& .0001 & 0.77051600758209 & 0.77051600898599 & & \\
&   .00005 & 0.77051600843322 & 0.77051600878460 & & \\
\hline\hline
E[2, 4, 6, \ldots, 34] & .0001 & 0.63347197012177 & 0.63347197028753 & & \\
&   .00005 & 0.63347197021161 & 0.63347197025258 & & \\
\hline\hline
E[1, \ldots,34] & .0001 & 0.98041962337899 & 0.98041962562238 & & \\
&   .00005 & 0.98041962476506 & 0.98041962532582 & &  \\
\hline 
\end{tabular}
\end{center}
\end{table}

\begin{table}[!ht]
\footnotesize
\caption{Computation of Hausdorff dimension $s$ of E[1,2]
using higher order piecewise polynomials.}
\label{tb:t2}
\begin{center}
\begin{tabular}{|c|c|c|c|c|c|c|}
\hline
degree & h  &  $s$ & error \\
\hline \hline 
1 & .01 & 0.531282991861209 & 2.49 e-06 \\
2 & .02 & 0.531280509905738 & 3.63 e-09  \\
4 & .04 & 0.531280506277707 & 5.07 e-13  \\
5 & .05 & 0.531280506277198 & 2.44 e-15 \\
\hline 
\end{tabular}
\end{center}
\end{table}
In the computations shown using higher order piecewise polynomials, since the
number of unknowns for a continuous, piecewise polynomial of degree $k$ on $n$
uniformly spaced subintervals of width $h$ is given by $k n + 1$, to get a
fair comparison, we have adjusted the mesh sizes so that each computation
involves the same number of unknowns. For this problem, the eigenfunction
$v_s$ is smooth and the computations show a dramatic increase in the accuracy
of the approximation as the degree of the approximating piecewise polynomial
is increased.

\begin{table}[!ht]
\footnotesize
\caption{Computation of Hausdorff dimension $s$ of E[2,4,6,8,10]
using piecewise cubic polynomials.}
\label{tb:t22}
\begin{center}
\begin{tabular}{|c|c|c|c|c|c|c|}
\hline
h  &  $s$ \\
\hline \hline 
0.1    &    0.517357031893604       \\
.05    &    0.517357031040157       \\  
.02    &    0.517357030941730       \\
.01    &    0.517357030937109       \\
.005   &    0.517357030937029       \\
.002   &    0.517357030937019       \\
.001   &    0.517357030937018       \\
\hline 
\end{tabular}
\end{center}
\end{table}


\subsection{An example with less regularity}
\label{subsec:lowreg}

For $0 \le \alam \le 1$, we consider the maps
\begin{equation}
\label{lrmaps}
\theta_1(x) = \frac{1}{3 + 2 \alam}(x + \alam x^{7/2}), \qquad
\theta_2(x) = \frac{1}{3 + 2 \alam}(x + \alam x^{7/2}) +
\frac{2 + \alam}{3 + 2 \alam},
\end{equation}
which map the unit interval to itself. Both these maps $\in C^3([0,1]$, but
$\notin C^4([0,1]$. We note that because of the lack of regularity, the
methods of \cite{Jenkinson-Pollicott} and \cite{Jenkinson} cannot be applied.
When $\alam=0$, these maps become
\begin{equation*}
\theta_1(x) = \frac{x}{3}, \qquad
\theta_2(x) = \frac{x}{3} + \frac{2}{3},
\end{equation*}
and the corresponding Cantor set has Hausdorff dimension $\ln 2/\ln 3
\hfill\break \approx 0.630929753571458$.  

Our computations, shown in Table~\ref{tb:t3}, are based on the following
result, which we shall prove in subsequent sections.
\begin{thm}
\label{thm:posev1dreg}
Let
\begin{equation*}
(L_s w)(x) = \sum_{j=1}^2 |\theta_j^{\prime}(x)|^s w(\theta_j(x)),
\end{equation*}
where $\theta_1$ and $\theta_2$ are given by \eqref{lrmaps}, and we have not
indicated the dependence on $\alam$ in our notation.
For all $s >0$, $L_s$ has a unique (up to normalization) strictly positive
$C^2$ eigenfunction $v_s$ with $L_s v_s = r_s v_s$, where $r_s >0$
and $r_s = r(L_s)$, the spectral radius of $L_s$.  Furthermore, the map
$s \mapsto r_s$ is strictly decreasing and continuous, and for all $x_1,
x_2 \in [0,1]$, we have the estimate
\begin{equation*}
0 < \frac{v_s^{\prime\prime}(x)}{v_s(x)} 
\le  \Big[s G_2(\alam) +\frac{ 2 s^2 C_1(\alam)^2 \kappa(\alam)}
{1 - \kappa(\alam)} + \frac{ s C_1(\alam) E_2(\alam)}{1 - \kappa(\alam)}\Big]
\big[1 - \kappa(\alam)^2\big]^{-1},
\end{equation*}
where $\kappa(\alam)$, $C_1(\alam)$, $E_2(\alam)$, $C_2(\alam)$, and
$G_2(\alam)$ are given by \eqref{3.44}, \eqref{3.45}, \eqref{3.47}, 
\eqref{3.48}, and \eqref{3.50}, respectively, 
and $\alam$ is as in \eqref{lrmaps}.  Finally, the Hausdorff dimension of the
Cantor set generated from the maps $\theta_1$ and $\theta_2$ is the unique
value of $s$ with $r_s= r(L_s) =1$.
\end{thm}

\begin{table}[!ht]
\caption{Computation of Hausdorff dimension $s$ of less regular examples.} 
\label{tb:t3}
\begin{center}
\begin{tabular}{|c|c|c|c|c|}
\hline
$\alam$    &   $h$  &  lower $s$ & upper $s$ & upper $s$ - lower $s$ \\
\hline 
$0.0$  &  .0001 &  $0.630929753571456$ & $0.630929753571458$ & $2.00e-15$ \\
$0.25$ &  .0001 &  $0.691029100877742$ & $0.691029110502742$ & $9.63e-09$ \\
$0.5$  &  .0001 &  $0.733474573000780$ & $0.733474622222678$ & $4.92e-08$ \\
$0.75$ &  .0001 &  $0.767207065889322$ & $0.767207292955631$ & $2.27e-07$  \\
$1.0$  &  .0001 &  $0.796726361744928$ & $0.796727861914648$ & $1.50e-06$ \\
\hline
\end{tabular}
\end{center}
\end{table}

\section{Existence of $C^m$ positive eigenfunctions}
\label{sec:exist}
In this section we shall describe some results concerning existence of $C^m$
positive eigenfunctions for a class of positive (in the sense of
order-preserving) linear operators.  We shall later indicate how one can often
obtain explicit bounds on derivatives of the positive eigenfunctions.  As
noted above, such estimates play a crucial role in our numerical method and
therefore in obtaining rigorous estimates of Hausdorff dimension for invariant
sets associated with iterated function systems.  The methods we shall describe
can also be applied to the important case of graph directed iterated function
systems, but for simplicity we shall restrict our attention in this paper to a
class of linear operators arising in the iterated function system case.

The starting point of our analysis is Theorem 5.5 in \cite{E}, which we
now describe for a simple case. If $H$ is a bounded open subset of
$\R$ and $m$ is a positive integer, $C^m(\bar H)$ will denote the
set of real-valued $C^m$ maps $w:H \to \R$ such that all 
derivatives $D^{k} w$ with $0 \le k \le m$ extend continuously
to $\bar H$. Here $D^k w = d^k w/dx^k$ and $C^m(\bar H)$ is a real Banach
space with $\|w\| = \sup\{|D^{k} w(x)|: x \in H, 0 \le k \le m\}$.

Let $\B$ denote a finite index set with $|\B| = p$.  For $\be \in \B$, we
assume
\begin{align*}
&\text{(H4.1)} \ \,
g_{\be} \in C^m(\bar H) \text{ for all } \be \in \B \text{ and }
g_{\be}(x) > 0 \text{ for all } x \in \bar H \text{ and all } \be \in \B.
\\
&\text{(H4.2)} \ \,
\theta_{\be}:H \to H \text{ is a } C^m \text{ map for all } \be \in \B.
\end{align*}

In (H4.1) and (H4.2), we always assume that $m \ge 1$.

We define $\Lambda: C^m(\bar H) \to C^m(\bar H)$ by
\begin{equation}
\label{1.2}
(\Lambda(w))(x) = \sum_{\be \in \B} g_{\be}(x) w(\theta_{\be}(x)).
\end{equation}
For integers $\mu \ge 1$, we define $\B_{\mu} := \{\w = (j_1, \ldots j_{\mu})
\, | \, j_k \in \B \text{ for } 1 \le k \ \le \mu\}$. For
$\w = (j_1, \ldots j_{\mu}) \in \B_{\mu}$, we define $\w_{\mu} = \w$,
$\w_{\mu -1} = (j_1, \ldots j_{\mu-1})$,
$\w_{\mu -2} = (j_1, \ldots j_{\mu-2})$, $\cdots$, $\w_1 = j_1$.  We define
\begin{equation*}
\theta_{\w_{\mu-k}}(x) = (\theta_{j_{\mu-k}} \circ \theta_{j_{\mu-k-1}} \circ
\cdots \circ \theta_{j_{1}})(x),
\end{equation*}
so
\begin{equation*}
\theta_{\w}(x):= \theta_{\w_{\mu}}(x) = 
(\theta_{j_{\mu}} \circ \theta_{j_{\mu-1}} \circ
\cdots \circ \theta_{j_{1}})(x).
\end{equation*}

For $\w \in \B_{\mu}$, we define $g_{\w}(x)$ inductively by 
$g_{\w}(x) = g_{j_1}(x)$
if $\w = (j_1) \in \B:=\B_1$,
$g_{\w}(x) = g_{j_2}(\theta_{j_1}(x)) g_{j_1}(x)$ if $\w = (j_1,j_2) \in \B_2$
and, for $\w = (j_1,j_2, \ldots j_{\mu}) \in \B_{\mu}$,
\begin{equation*}
g_{\w}(x) = g_{j_\mu}(\theta_{\w_{j_{\mu-1}}}(x)) g_{\w_{\mu -1}}(x).
\end{equation*}

If is not hard to show (see \cite{A}, \cite{Bourgain-Kontorovich}, \cite{E})
that
\begin{equation}
\label{1.6}
(\Lambda^{\mu}(w))(x) = \sum_{\w \in \B_{\mu}} g_{\w}(x) w(\theta_\w(x)).
\end{equation}

It is easy to prove (see \cite{E}) that $\Lambda$ defines a bounded linear map
of $C^m(\bar H) \to C^m(\bar H)$.  We shall let $\hat \Lambda$ denote the
complexification of $\Lambda$ and let $\sigma(\hat \Lambda)$ denote the
spectrum of $\hat \Lambda$.  We shall define $\sigma(\Lambda) = \sigma(\hat
\Lambda)$. If all the functions $g_{\be}$ and $\theta_{\be}$ are $C^N$,
then we can consider $\Lambda$ as a bounded linear operator $\Lambda_m:
C^m(\bar H) \to C^m(\bar H)$ for $1 \le m \le N$, but one should note that in
general $\sigma(\Lambda_m)$ will depend on $m$.

To obtain a useful theory for $\Lambda$, we need a further crucial assumption.

(H4.3)  There exists a positive integer $\mu$ and a constant $\kappa <1$ such
that for all $\w \in \B_{\mu}$ and all $x,y \in H$,
$|\theta_\w(x) - \theta_\w(y)| \le \kappa |x-y|$.

If we define $c = \kappa^{1/\mu} <1$, it follows from (H4.3) that there
exists a constant $M$ such that for all $\w \in B_{\nu}$ and all $\nu \ge 1$,
\begin{equation}
\label{1.8}
|\theta_\w(x) - \theta_\w(y)| \le M c^{\nu} |x-y| \quad \forall x,y \in H.
\end{equation}

The following theorem is a special case of Theorem 5.5 in \cite{E}.

\begin{thm}
\label{thm:1.1} 
Let $H$ be a bounded open subset of $\R$, which is a finite union of
open intervals. Let $X = C^m(\bar H)$ and assume that (H4.1), (H4.2),
and (H4.3) are satisfied (where $m \ge 1$ in (H4.1) and (H4.2)) and
that $\Lambda:X \to X$ is given by \eqref{1.2}.  If $Y= C(\bar H)$,
the Banach space of real-valued continuous functions $w: \bar H \to
\R$ and $L:Y \to Y$ is defined by \eqref{1.2}, then $r(L) = r(\Lambda)
>0$, where $r(L)$ denotes the spectral radius of $L$ and $r(\Lambda)$
denotes the spectral radius of $\Lambda$.  If $\rho(\Lambda)$ denotes
the essential spectral radius of $\Lambda$ (see
\cite{B},\cite{A},\cite{N-P-L}, and \cite{L}), then $\rho(\Lambda) \le
c^m r(\Lambda)$ where $c= \kappa^{1/\mu}$ is as in \eqref{1.8}.  There
exists $v \in X$ such that $v(x) >0$ for all $x \in \bar H$ and
\begin{equation*}
\Lambda(v) = r v, \qquad r = r(\Lambda).
\end{equation*}
There exists $r_1 < r$ such that if $\xi \in \sigma(\Lambda)
\setminus\{r\}$,
then $|\xi| \le r_1$; and $r = r(\Lambda)$ is an isolated point of
$\sigma(\Lambda)$ and an eigenvalue of algebraic multiplicity 1. If $u \in X$
and $u(x) >0$ for all $x \in \bar H$, there exists a real number $s_u >0$ such
that
\begin{equation}
\label{1.10}
\lim_{k \rightarrow \infty}\left(\frac{1}{r} \Lambda\right)^k (u) = s_u v,
\end{equation}
where the convergence in \eqref{1.10} is in the $C^m$ topology on $X$.
\end{thm}

\begin{remark}
  \label{rem:1.2} If $l$ is an integer satisfying $0 \le l \le m$, where $m \ge
  1$ is as in (H4.1) and (H4.2), it follows from \eqref{1.10} that
\begin{equation}
\label{1.11}
\lim_{k \rightarrow \infty} \left(\frac{1}{r}\right)^k D^{l} \Lambda^k(u) 
=s_u D^{l} v,
\end{equation}
and
\begin{equation}
\label{1.12}
\lim_{k \rightarrow \infty} \left(\frac{1}{r}\right)^k \Lambda^k(u) 
=s_u v,
\end{equation}
where the convergence in \eqref{1.11} and \eqref{1.12} is in the topology
of $C(\bar H)$, the Banach space of continuous functions $w: \bar H \to \R$.
\end{remark}
It follows from \eqref{1.11} and \eqref{1.12} that for any integer $l$ with
$0 \le l \le m$,
\begin{equation}
\label{1.13}
\lim_{k \rightarrow \infty} \frac{(D^{l} \Lambda^k(u))(x)}
{\Lambda^k(u)(x)} = \frac{(D^{l} (v))(x)}
{v(x)},
\end{equation}
where the convergence in \eqref{1.13} is uniform in $x \in \bar H$.
If we choose $u(x) =1$ for all $x \in \bar H$, it follows from \eqref{1.6}
that for all integers $l$ with $0 \le l \le m$, we have
\begin{equation}
\label{1.14}
\lim_{k \rightarrow \infty}  \frac{D^{l} (\sum_{\w \in \B_k} g_\w(x))}
{\sum_{\w \in \B_k} g_\w (x)} = \frac{D^{l} v(x)}{v(x)},
\end{equation}
where the convergence in \eqref{1.14} is uniform in $x \in \bar H$. We shall
use \eqref{1.14} in our further work to obtain explicit bounds on 
$\sup\left\{|D^{l} v(x)|/v(x): x \in \bar H\right\}$.

\section{Estimates for derivatives of $v_s$: Mappings of form
\eqref{intro1.2}}
\label{sec:1d-deriv}
Throughout this section, we shall assume for simplicity that
$H=(a_1,a_2)$ is a bounded,
open interval, although it is frequently natural to take $H$ to be the finite
union of disjoint intervals. $\B$ will denote a finite index set.
For $\be \in \B$ and some integer $m \ge 1$, we assume

\noindent (H5.1:) For each $\be \in \B$, $g_{\be} \in C^m(\bar H)$,
$\theta_{\be} \in C^m(\bar H)$, $g_{\be}(x) >0$ for all $x \in \bar H$ and
$\theta_{\be}(H) \subset H$.  There exist an integer $\mu \ge 1$ and a real
number $\kappa <1$ such that for all $\omega \in \B_{\mu}:= \{(\be_1,
\be_2, \cdots, \be_{\mu}) \, | \, \be_j \in \B$ for $1 \le j \le \mu\}$
and for all $x,y \in \bar H$, $|\theta_{\omega}(x) - \theta_{\omega}(y)| \le
\kappa |x-y|$, where $\theta_{\omega}:= \theta_{\be_{\mu}} \circ
\theta_{\be_{\mu-1}} \circ \cdots \circ \theta_{\be_1}$ for $\omega =
(\be_1, \be_2, \cdots, \be_{\mu}) \in \B_{\mu}$.

As in Section~\ref{sec:exist}, we define $Y = C(\bar H)$ and $X_m = C^m(\bar
H)$.  Assuming (H5.1), we define for $s \ge 0$, a bounded linear operator
$L_s:Y \to Y$ by
\begin{equation}
\label{3.1} (L_s w)(x) = \sum_{\be \in \B} [g_{\be}(x)]^s w(\theta_{\be}(x)).
\end{equation} 
As in Section~\ref{sec:exist}, $L_s(X_m) \subset X_m$ and $L_s |_{X_m}$
defines a bounded linear map of $X_m$ to $X_m$ which we denote by $\Lambda_s$.
Theorem~\ref{thm:1.1} is now directly applicable (replace $g_{\be}(x)$ in
Theorem~\ref{thm:1.1} by $[g_{\be}(x)]^s$) and yields information about
$\sigma(\Lambda_s)$.  In particular, $r(L_s) = r(\Lambda_s) >0$ and there
exists a unique (to within normalization) strictly positive, $C^m$
eigenfunction $v_s$ of $\Lambda_s$ with eigenvalue $\lambda_s = r(\Lambda_s)$.

If $\omega = (\be_1, \be_2, \ldots, \be_p) \in \B_p$, recall that
we define $g_{\omega}(x)$ by
\begin{equation*}
g_{\omega}(x) = g_{\be_p}(\theta_{\be_{p-1}} \circ \theta_{\be_{p-2}}
\circ \cdots \circ \theta_{\be_{1}}(x)) \cdots
g_{\be_3}((\theta_{\be_{2}} \circ \theta_{\be_{1}})(x))
g_{\be_2}((\theta_{\be_{1}}(x)) g_{\be_1}(x),
\end{equation*}
and
\begin{equation}
\label{3.2}
(L_s^p w)(x) = \sum_{\omega \in \B_p} [g_{\omega} (x)]^s w(\theta_{\omega}(x)).
\end{equation}

Notice that $L_s^p$ is of the same form as $L_s$ and Theorem~\ref{thm:1.1}
is also directly applicable to  $L_s^p$.  Since $v_s$ is also an eigenfunction
of $L_s^p$, we can also work with \eqref{3.2} instead of \eqref{3.1}:
$\B_p$ is an index set corresponding to $\B$, $g_{\omega}$, $\omega \in
\B_p$, corresponds to $g_{\be}$, $\be \in \B$, and $\theta_{\omega}$,
$\omega \in \B_p$, corresponds to $\theta_{\be}$, $\be \in \B$.

If $m$ is as in (H5.1) and $k$ is a positive integer with $k \le m$,
we define $D = d/dx$, so $(Df)(x) = f^{\prime}(x)$ and
$(D^kf)(x) = f^{(k)}(x)$.  We are interested in obtaining estimates for
\begin{equation}
\label{3.3}
\sup \{|D^k v_s(x)|/v_s(x) : x \in \bar H\}.
\end{equation}
We note that the estimates we shall give below can
be refined as in Section 6 of \cite{hdcomplt}, but for simplicity we shall
omit these refinements.



First observe that Hypothesis (H5.1) implies that there exist constants $M >0$
and $c = \kappa^{1/\mu}$, (so $c <1$), such that for
all integers $\nu \ge 1$ and all $\w \in \B_{\nu}$, \eqref{1.8} is satisfied.


If $\mu$ is as in (H5.1), we define a constant $C_1$ by
\begin{equation}
\label{3.7}
C_1 = \sup\Big\{\frac{|g_{\w}^{\prime}(x)|}{g_{\w}(x)}: \w \in \B_{\mu}, x \in \bar H
\Big\}.
\end{equation}

A calculation shows that for all $\w \in \B_{\nu}$, $\nu \ge 1$,
\begin{equation}
\label{3.8}
\frac{D [g_{\w}(x)^s]}{[g_{\w}(x)]^s}
= s \frac{g_{\w}^{\prime}(x)}{g_{\w}(x)},
\end{equation}
so
\begin{equation*}
\sup \Big\{\frac{|D [g_{\w}(x)^s]|}{[g_{\w}(x)]^s}: 
\w \in \B_{\mu}, x \in \bar H \Big\} = s C_1.
\end{equation*}

We begin by considering \eqref{3.3} for the case $k=1$. In our applications,
we shall only need the case $s >0$, so we shall restrict our attention to
this case.
\begin{thm}
\label{thm:3.1}
Assume that (H5.1) is satisfied, let $\mu$, $m$, and $\kappa$ be as in 
(H5.1) and let $C_1$ be as in \eqref{3.7},  For $s >0$, let
$v_s$ denote the unique (to within normalization) strictly positive
eigenfunction of $\Lambda_s := L_s|_{X_m}$.  Then we have
\begin{equation}
\label{3.10}
\sup \Big\{\frac{|v_s^{\prime}(x)|}{v_s(x)} : x \in \bar H\Big\} 
\le \frac{C_1 s}{1- \kappa}:= M_1.
\end{equation}
If $\delta \in \{0,1\}$ and $(-1)^{\delta} g_\w^{\prime}(x)/g_\w(x))\le 0$
for all $\w \in \B_{\nu}$, all $\nu \ge 1$ and all $x \in \bar H$, then
$(-1)^{\delta} v_s^{\prime}(x) \le 0$ for all $x \in \bar H$ and all $s >0$.
\end{thm}
\begin{proof}
Recall from Section~\ref{sec:exist} that $v_s$ is (after normalization)
also the unique eigenfunction of $\Lambda_s^{\mu}$ with eigenvalue $r^{\mu}$,
where $r = r(\Lambda_s)$ and $r^{\mu}$ is the spectral radius of 
$\Lambda_s^{\mu}$. Define $\hat M_1$ by
\begin{equation*}
\hat M_1 = \sup \Big\{\frac{|v_s^{\prime}(x)|}{v_s(x)} : x \in \bar H \Big\}.
\end{equation*}
We shall prove that $\hat M_1 \le M_1$. For notational
convenience we write for $\w \in \B_{\mu}$
\begin{equation*}
f_{\w}(x) = [g_{\w}(x)]^s v_s(\theta_{\w}(x)).
\end{equation*}
Then we see that
\begin{equation}
\label{3.11}
\Big|\frac{\lambda_s v_s^{\prime}(x)}{\lambda_s v_s(x)} \Big|
= \frac{|\lambda_s v_s^{\prime}(x)|}{\lambda_s v_s(x)}
= \frac{\Big| \sum_{\w \in \B_{\mu}} f_{\w}^{\prime}(x)\Big|}
{\sum_{\w \in \B_{\mu}} f_{\w}(x)}
\le \frac{\sum_{\w \in \B_{\mu}} |f_{\w}^{\prime}(x)|}
{\sum_{\w \in \B_{\mu}} f_{\w}(x)}.
\end{equation}

A calculation shows that
\begin{equation*}
\frac{f_{\w}^{\prime}(x)}{f_{\w}(x)} = s \frac{g_{\w}^{\prime}(x)}{g_{\w}(x)}
+ \frac{v_s^{\prime}(\theta_{\w}(x)) \theta_{\w}^{\prime}(x)}
{v_s(\theta_{\w}(x))},
\end{equation*}
so
\begin{equation*}
\frac{|f_{\w}^{\prime}(x)|}{f_{\w}(x)} \le s C_1 + \hat M_1 \kappa
\end{equation*}
and
\begin{equation}
\label{3.12}
\frac{\sum_{\w \in \B_{\mu}} |f_{\w}^{\prime}(x)|}
{\sum_{\w \in \B_{\mu}} f_{\w}(x)} \le s C_1 + \hat M_1 \kappa.
\end{equation}

Taking the maximum of the left hand side of \eqref{3.11}, we deduce
from \eqref{3.11} and \eqref{3.12} that
\begin{equation}
\label{3.13}
\hat M_1 \le sC_1 + \hat M_1 \kappa
\end{equation}
and \eqref{3.13} implies that
\begin{equation*}
\hat M_1 \le sC_1/(1- \kappa) = M_1.
\end{equation*}
\end{proof}

Throughout the remainder of this section, $C_1$ will be as in \eqref{3.7}
and $M_1$ will be as in \eqref{3.10}.
Assuming that $m$ and $\mu$ are as in (H5.1) and $m \ge 2$, it will also
be convenient to define constants $C_2$, $E_2$, and $K_2$ by
\begin{equation}
\label{3.15}
C_2 = \sup \Big\{\frac{|g_{\w}^{\prime\prime}(x)|}{g_{w}(x)}: 
\w \in \B_{\mu}, x \in \bar  H\Big\},
\end{equation}
\begin{equation*}
E_2 = \sup \Big\{|\theta_{\w}^{\prime\prime}(x)|: \w \in \B_{\mu}, x \in \bar
 H\Big\},
\end{equation*}
\begin{equation*}
K_2 = \sup \Big\{\frac{| g_{\w}^{\prime\prime}(x) g_{w}(x) 
- (1-s) [g_{\w}^{\prime}(x)]^2|}
{[g_{\w}(x)]^2}: \w \in \B_{\mu}, x \in \bar
 H\Big\}.
\end{equation*}
Notice that we always have the estimate
$K_2 \le C_2 + |1-s|C_1^2$, but sometimes more precise estimates for
$K_2$ can be obtained.

\begin{thm}
\label{thm:3.2}
Assume that (H5.1) is satisfied with $m \ge 2$ and let $\mu$, $m$, and
$\kappa$ be as in (H5.1). Assume that $s >0$ and let $C_1$, $M_1$, $C_2$,
$E_2$, and $K_2$ be as defined above. Let $v_s$ denote the unique (to within
normalization) strictly positive eigenfunction of $\Lambda_s: X_m \to X_m$
with eigenvalue $r(\Lambda_s)$.  Then we have
\begin{equation}
\label{3.18}
\sup\Big\{\frac{|v_s^{\prime\prime}(x)|}{v_s(x)}: x \in \bar H\Big\} \le M_2,
\end{equation}
where
\begin{equation}
\label{3.19}
M_2: = (s K_2 + 2s C_1M_1 \kappa + M_1 E_2)/(1 - \kappa^2).
\end{equation}
\end{thm}
\begin{proof}
As in the proof of Theorem~\ref{thm:3.1}, for $\w \in \B_{\mu}$, let
$f_{\w}(x) = [g_{\w}(x)]^s v_s(\theta_{\w}(x))$ and observe that
\begin{equation}
\label{3.20}
\Big|\frac{\lambda_s v_s^{\prime\prime}(x)}{\lambda_s v_s(x)} \Big|
= \frac{|\lambda_s v_s^{\prime\prime}(x)|}{\lambda_s v_s(x)}
= \frac{\Big| \sum_{\w \in \B_{\mu}} f_{\w}^{\prime\prime}(x)\Big|}
{\sum_{\w \in \B_{\mu}} f_{\w}(x)}
\le \frac{\sum_{\w \in \B_{\mu}} |f_{\w}^{\prime\prime}(x)|}
{\sum_{\w \in \B_{\mu}} f_{\w}(x)}.
\end{equation}
A calculation shows that
\begin{multline}
\label{3.21}
\frac{f_{\w}^{\prime\prime}(x)}{f_{\w}(x)} = \Big[s(s-1) 
\Big(\frac{g_{\w}^{\prime}(x))}{g_{\w}(x)}\Big)^2 +s 
\frac{g_{\w}^{\prime\prime}(x)}{g_{\w}(x)}\Big]
+ 2s \Big[\frac{g_{\w}^{\prime}(x)}{g_{\w}(x)}
\frac{v_s^{\prime}(\theta_{\w}(x))
  \theta_{\w}^{\prime}(x)}{v_s(\theta_{\w}(x))}\Big]
\\
+  \Big[\frac{v_s^{\prime\prime}(\theta_{\w}(x))}{v_s(\theta_{\w}(x))}
  \Big(\theta_{\w}^{\prime}(x)\Big)^2
+ \frac{v_s^{\prime}(\theta_{\w}(x))}{v_s(\theta_{\w}(x))}
  \Big(\theta_{\w}^{\prime\prime}(x)\Big) \Big].
\end{multline}
If we define $\hat M_2 = \sup\{|v_s^{\prime\prime}(x)|/v_s(x): x \in \bar H\}$,
we obtain from
\eqref{3.21} that
\begin{equation}
\label{3.22}
\frac{|f_{\w}^{\prime\prime}(x)|}{f_{\w}(x)} \le s K_2 + 2s[C_1 M_1 \kappa]
+ \hat M_2 \kappa^2 + M_1 E_2,
\end{equation}
and using \eqref{3.22} and \eqref{3.20}, we see that
\begin{equation*}
\hat M_2 \le s K_2 + 2s C_1 M_1 \kappa + M_1 E_2 + \hat M_2 \kappa^2,
\end{equation*}
which implies that $\hat M_2 \le M_2$ (defined in \eqref{3.19}).
\end{proof}

\begin{remark}
\label{rem:3.1}
If one has obtained bounds for $\sup \{|D^jv_s(x)|/v_s(x): x \in \bar H\}$
for $1 \le j \le k$, it is not hard to show that the kind of argument
in the proof of Theorem~\ref{thm:3.2} can be used to estimate
$\sup \{|D^{k+1}v_s(x)|/v_s(x): x \in \bar H\}$.
\end{remark}
Rather than give a formal proof of the general case, we shall restrict
ourselves here to obtaining an estimate for
$\sup \{|D^{3}v_s(x)|/v_s(x): x \in \bar H\}$.
To state our theorem, it will be convenient to introduce further
constants $C_3$, $E_3$, and $K_3$:
\begin{equation*}
C_3 = \sup \Big\{\frac{|D^3 g_{\w}(x)|}{g_{w}(x)}: \w \in \B_{\mu}, x \in \bar
 H\Big\},
\end{equation*}
\begin{equation*}
E_3 = \sup \Big\{|D^3 \theta_{\w}(x)|: \w \in \B_{\mu}, x \in \bar
 H\Big\},
\end{equation*}
\begin{equation*}
K_3 
= \sup \Big\{\frac{|(s-1)(s-2) [g_{\w}^{\prime}(x)]^3
+ 3(s-1) g_{\w}(x) g_{\w}^{\prime}(x) g_{\w}^{\prime\prime}(x) 
+ [g_{\w}(x)]^2 g_{\w}^{\prime\prime\prime}(x))|}
{[g_{\w}(x)]^3} \Big\},
\end{equation*}
where the supremum is taken over $\w \in \B_{\mu}$ and $x \in \bar  H$.

A crude estimate for $K_3$ in terms of $C_1$, $C_2$, and $C_3$ can be
given:
\begin{equation*}
K_3 \le |s-1| |s-2| C_1^3 + 3 |s-1| C_1 C_2 + C_3.
\end{equation*}
However, better estimates are frequently available.

\begin{thm}
\label{thm:3.3}
Assume that (H5.1) is satisfied with $m \ge 3$ and let $\mu$, $m$, and
$\kappa$ be as in (H5.1). Assume that $s >0$ and let $C_1$, $M_1$, $C_2$,
$E_2$, $K_2$, $M_2$, $C_3$, $E_3$, and $K_3$ be as defined above.  If $v_s$ is
the normalized strictly positive eigenfunction of $\Lambda_s: X_m \to X_m$
with eigenvalue $r(\Lambda_s)$, then we have
\begin{equation}
\label{3.27}
\sup\Big\{\frac{|D^3 v_s(x)|}{v_s(x)}: x \in \bar H\Big\} \le M_3,
\end{equation}
where
\begin{multline*}
(1- \kappa^3) M_3 
\\
= (s K_3 + 3s K_2 M_1 \kappa + 3sC_1(M_2 \kappa^2 + M_1 E_2)
+ 3 M_2 \kappa E_2 + M_1 E_3) := S.
\end{multline*}
\end{thm}

\begin{proof}
  Again set $f_{\w}(x) = [g_{\w}(x)]^s v_s(\theta_{\w}(x))$ and define $\hat
  M_3 = \sup\{|D^3v_s(x)|/v_s(x): x \in \bar H\}$.  As in the proof of
  Theorem~\ref{thm:3.2}, we find that
\begin{equation}
\label{3.28}
\frac{D^3 v_s(x)}{v_s(x)}
= \frac{\sum_{\w \in \B_{\mu}} D^3f_{\w}(x)}
{\sum_{\w \in \B_{\mu}} f_{\w}(x)}.
\end{equation}
A calculation shows that
\begin{multline*}
\frac{D^3f_{\w}(x)}{f_{\w}(x)} = \frac{D^3[g_{\w}(x)^s]}{g_{\w}(x)^s}
+ 3 \frac{D^2[g_{\w}(x)^s]}{g_{\w}(x)^s}
\frac{D[v_s(\theta_{\w}(x))]}{v_s(\theta_{\w}(x))}
\\
+ 3 \frac{D[g_{\w}(x)^s]}{g_{\w}(x)^s}
\frac{D^2[v_s(\theta_{\w}(x))]}{v_s(\theta_{\w}(x))}
+ \frac{D^3[v_s(\theta_{\w}(x))]}{v_s(\theta_{\w}(x))}.
\end{multline*}
Further tedious calculations give
\begin{gather*}
\frac{|D^3[g_{\w}(x)^s]|}{g_{\w}(x)^s} \le s K_3,
\\
3 \frac{|D^2[g_{\w}(x))^s]|}{g_{\w}(x)^s}
\frac{|D[v_s(\theta_{\w}(x))]|}{v_s(\theta_{\w}(x))} \le 3 s K_2 M_1 \kappa,
\\
3 \frac{|D[g_{\w}(x)^s]|}{g_{\w}(x)^s}
\frac{|D^2[v_s(\theta_{\w}(x))]|}{v_s(\theta_{\w}(x))} \le
3s C_1(M_2 \kappa^2 + M_1 E_2),
\\
\frac{|D^3[v_s(\theta_{\w}(x))]|}{v_s(\theta_{\w}(x))}
\le 3 M_2 \kappa E_2 + M_1 E_3 + \hat M_3 \kappa^3.
\end{gather*}
It follows that $|D^3 f_{\w}(x)| \le (S + \hat M_3 \kappa^3) f_{\w}(x)$, where
$S$ is as in the statement of Theorem~\ref{thm:3.3}.  This proves that the
absolute value of the right side of \eqref{3.28} is less than or equal to $S +
\hat M_3 \kappa^3$.  Taking the supremum of the left hand side of \eqref{3.28}
for $x \in \bar H$ gives $\hat M_3 \le S + \hat M_3 \kappa^3$, which implies
\eqref{3.27}.
\end{proof}

Theorems~\ref{thm:3.1} -- \ref{thm:3.3} are crude.  If one has more
information about the coefficients $g_{\be}(\cdot)$ and the maps
$\theta_{\be}(\cdot)$, $\be \in \B$, one can frequently obtain much sharper
results.  An example is provided by the following theorem.

\begin{thm}
\label{thm:3.4}
Assume that (H5.1) is satisfied with $m \ge 2$.
Assume also that
$\theta_{\be}^{\prime}(u) \ge 0$, $\theta_{\be}^{\prime\prime}(u) \ge 0$,
$g_{\be}^{\prime}(u) \ge 0$, $g_{\be}^{\prime\prime}(u) \ge 0$, and
\begin{equation}
\label{3.29}
g_{\be}^{\prime\prime}(u) g_{\be}(u) - (1-s)[g_{\be}^{\prime}(u)]^2
\ge 0
\end{equation}
for all $\be \in \B$, for all $u \in H$, and for a given positive real number
$s$. If $v_s$ is the strictly positive $C^m$ eigenfunction of
$\Lambda_s$, it follows that for all $u \in \bar H$
\begin{equation*}
v_s^{\prime}(u) \ge 0 \text{ and } v_s^{\prime\prime}(u) \ge 0.
\end{equation*}
If, in addition, there exists a set $F \subset \bar H$ (possibly empty) such
that for all $u \in \bar H \setminus F$ and all $\be \in \B$, 
$g_{\be}^{\prime}(u) > 0$ and strict inequality holds in \eqref{3.29}, then
for all $u \in \bar H \setminus F$,
\begin{equation*}
v_s^{\prime}(u) >  0 \text{ and } v_s^{\prime\prime}(u) >  0.
\end{equation*}
\end{thm}
\begin{proof}
  For $\nu \ge 1$, let $\w =(b_1, b_2, \cdots,, b_{\nu})$ denote a fixed
  element of $\B_{\nu}$ and for $0 \le k \le \nu$, define $\xi_0(x) = x$,
  $\xi_1(x) = \theta_{b_1}(x)$ and generally $\xi_k(x) = (\theta_{b_k} \circ
    \theta_{b_{k-1}} \circ \cdots \circ \theta_{b_1}(x))$.  We leave to the
      reader the simple proof that $\xi_k^{\prime}(x) \ge 0$ and
      $\xi_k^{\prime\prime}(x) \ge 0$ for all $x \in \bar H$ and $0 \le k \le
      \nu$.  Using \eqref{3.8}, a straightforward calculation yields
\begin{equation}
\label{3.32}
\frac{D [g_{\w}(x)^s]}{ g_{\w}(x)^s}
 = s \frac{g_{\w}^{\prime}(x)}{g_{\w}(x)}
= s \sum_{k=0}^{\nu-1}
\frac{g_{b_{k+1}}^{\prime}(\xi_k(x)) \xi_k^{\prime}(x)}
{g_{b_{k+1}}(\xi_k(x))} \ge s \frac{g_{b_1}^{\prime}(x)}{g_{b_1}(x)} \ge 0.
\end{equation}
Using \eqref{1.14} and taking the limit as $\nu \rightarrow \infty$, we
conclude that $v_s^{\prime}(x)/v_s(x) \ge 0$ for all $x \in \bar H$. If, in
addition, there exists a set $F$ as in the statement of Theorem~\ref{thm:3.4}
and if $x \notin F$, it follows that
\begin{equation*}
\inf\Big\{s \frac{g_{\be}^{\prime}(x)}{g_{\be}(x)} : \be \in \B \Big\} := s
\delta_1(x) >0,
\end{equation*}
so \eqref{3.32} then implies that
\begin{equation*}
\frac{D[g_{\w}(x)^s]}{g_{\w}(x)^s} \ge s \delta_1(x).
\end{equation*}
Again using \eqref{1.14} and letting $\nu \rightarrow \infty$, we conclude
that $v_s^{\prime}(x) \ge s  \delta_1(x) >0$ for all $x \in \bar H \setminus F$.

For $\w = (b_1, b_2, \cdots, b_{\nu}) \in \B_{\nu}$, we obtain from
\eqref{3.32} that
\begin{equation*}
D [g_{\w}(x)^s] = s g_{\w}(x)^s \sum_{j=0}^{\nu-1}
\frac{g_{b_{j+1}}^{\prime}(\xi_j(x)) \xi_j^{\prime}(x)}
{g_{b_{j+1}}(\xi_j(x))} := s g_{\w}(x)^s \sum_{j=0}^{\nu-1} T_j(x)
\end{equation*}
and $D[g_{\w}(x)] = g_{\w}(x)  \sum_{j=0}^{\nu-1} T_j(x)$.  A calculation
now gives
\begin{multline}
\label{3.33}
D^2[g_{\w}(x)^s]
= s D\Big[g_{\w}(x)^s \sum_{j=0}^{\nu-1} T_j(x)\Big]
\\
= s \Big[s g_{\w}(x)^s \Big(\sum_{j=0}^{\nu-1} T_j(x)\Big)^2
+ g_{\w}(x)^s \sum_{j=0}^{\nu-1} D(T_j(x))\Big].
\end{multline}
Because $T_j(x) \ge 0$ for all $x \in \bar H$ and $1 \le j \le \nu-1$,
\begin{equation*}
\Big(\sum_{j=0}^{\nu-1} T_j(x)\Big)^2 \ge \sum_{j=0}^{\nu-1} (T_j(x))^2
= \sum_{j=0}^{\nu-1} \frac{\big[g_{b_{j+1}}^{\prime}(\xi_j(x))\big]^2 
(\xi_j^{\prime}(x))^2} {[g_{b_{j+1}}(\xi_j(x))]^2}.
\end{equation*}
A calculation gives
\begin{multline}
\label{3.35}
\sum_{j=0}^{\nu-1} D(T_j(x))
= \sum_{j=0}^{\nu-1} \frac{\Big[g_{b_{j+1}}^{\prime\prime}(\xi_j(x))
  (\xi_j^{\prime}(x))^2
+ g_{b_{j+1}}^{\prime}(\xi_j(x))   \xi_j^{\prime\prime}(x)\Big]g_{b_{j+1}}(\xi_j(x))}
{[g_{b_{j+1}}(\xi_j(x))]^2}
\\
-  \sum_{j=0}^{\nu-1} [T_j(x)]^2.
\end{multline}
Combining \eqref{3.33} -- \eqref{3.35} and noticing that all terms in the
summation are nonnegative, we find that
\begin{multline}
\label{3.36}
D^2[g_{\w}(x)^s] \ge s [g_{\w}(x)]^s
\sum_{j=0}^{\nu-1} \Big(g_{b_{j+1}}^{\prime\prime}(\xi_j(x))
g_{b_{j+1}}(\xi_j(x)) - (1-s) [g_{b_{j+1}}^{\prime}(\xi_j(x))]^2 \Big)
\\
\cdot (\xi_j^{\prime}(x))^2 \big[g_{b_{j+1}}(\xi_j(x))\big]^{-2}.
\end{multline}

Since we assume that $g_b^{\prime\prime}(u) g_b(u) - (1-s) [g_b^{\prime}(u)]^2
\ge 0$ for all $u \in \bar H$ and $b \in \B$, we find that for all $\w \in
\B_{\nu}$, $x \in \bar H$,
$D^2[g_{\w}(x)^s]  \ge 0$.
Letting $\nu \rightarrow \infty$ and using \eqref{1.14}, we derive
that $v_s^{\prime\prime}(x) \ge 0$ for all $x \in \bar H$.

If a set $F \subset H$ exists such that strict inequality holds
in \eqref{3.29} for all $b \in \B$ and all $x \in \bar H \setminus F$,
then by only taking the term $j=0$ in the summation in \eqref{3.36},
we find that there is a number $\delta_2(x;s)>0$ for 
$x \in \bar H \setminus F$ and $s >0$ such that
\begin{equation*}
\frac{D^2[g_{\w}(x)^s]}{g_{\w}(x)^s} \ge \delta_2(x;s).
\end{equation*}
Again, using \eqref{1.14} and letting $\nu \rightarrow \infty$, this implies
that for $x \in \bar H \setminus F$,
\begin{equation*}
\frac{v_s^{\prime\prime}(x)}{v_s(x)} \ge \delta_2(x;s) >0,
\end{equation*}
which completes the proof.
\end{proof}

\begin{remark}
\label{rem:3.2}
An examination of the proof of Theorem~\ref{thm:3.4} shows that we have
proved that for all $x \in \bar H$, for all $\nu \ge 1$, and for all
$\w \in B_{\nu}$, $\theta_{\w}^{\prime}(x) \ge 0$,  
$\theta_{\w}^{\prime\prime}(x) \ge 0$, $g_{\w}^{\prime}(x) \ge 0$, 
$g_{\w}^{\prime\prime}(x) \ge 0$, and $D^2[g_{\w}(x)^s] \ge 0$. Because
\begin{equation*}
D^2[g_{\w}(x)^s] = s g_{\w}(x)^{s-2}\Big(g_{\w}^{\prime\prime}(x) g_{\w}(x)
- (1-s) [g_{\w}^{\prime}(x)]^2\Big),
\end{equation*}
we also see that
$g_{\w}^{\prime\prime}(x) g_{\w}(x) - (1-s) [g_{\w}^{\prime}(x)]^2 \ge 0$ for
all $\w \in B_{\nu}$, all $\nu \ge 1$, and all $x \in \bar H$. If the
constants $C_1$, $C_2$, $M_1$, $\kappa$, and $E_2$ are defined as above
in this section, one obtains immediately that for all $x \in \bar H$,
\begin{equation*}
0 \le \frac{v_s^{\prime}(x)}{v_s(x)} \le \frac{C_1 s}{1- \kappa}.
\end{equation*}
An examination of the proof of Theorem~\ref{thm:3.2} yields the following
refinement of \eqref{3.18}-\eqref{3.19}.
\begin{equation}
\label{3.39}
0 \le \frac{v_s^{\prime\prime}(x)}{v_s(x)} 
\le  \Big[ s G_2 + 2 s^2 C_1^2 \frac{\kappa}
{1- \kappa} + s C_1 E_2 \frac{1}{1- \kappa}\Big]
\Big[ \frac{1}{1- \kappa^2}\Big],
\end{equation}
where
\begin{equation}
\label{3.40}
G_2 = \max\Big\{
\frac{g_{\w}^{\prime\prime}(x) g_{\w}(x)
- (1-s) [g_{\w}^{\prime}(x)]^2}{g_{\w}(x)^2}: \w \in \B_{\mu}, x \in \bar H
\Big\}.
\end{equation}
\end{remark}

{\bf Example:} To illustrate the methods of this section, we consider a simple
example which nevertheless has some interest because of a failure of
smoothness which makes techniques in \cite{Jenkinson-Pollicott} inapplicable.
We shall always assume that $0 \le \alam \le 1$ and define
\begin{equation*}
\theta_1(x) = \frac{1}{3 + 2 \alam}(x + \alam x^{7/2}), \qquad
\theta_2(x) = \theta_1(x) + \frac{2 + \alam}{3 + 2 \alam},
\end{equation*}
so $\theta_j:[0,1] \to [0,1]$, $\theta_1(0) =0$, and $\theta_2(1) = 1$.  For
simplicity we suppress the dependence of $\theta_j(x)$ on $\alam$ in our
notation. If $\B=\{1,2\}$ and $\alam >0$ and $\w = (j_1, j_2, \ldots, j_{\nu})
\in \B_{\nu}$, notice that $D^3 \theta_{\w}(x)$ is defined and H\"older
continuous for all $x \in [0,1]$; but if $j_1 =1$, $D^4 \theta_{\w}(x)$ is not
defined at $x=0$.  Using that $0 \le \alam \le 1$, one can check that $0 <
\theta_j^{\prime}(x) <1$ for $0 \le x \le 1$; and it follows that there exists
a unique compact, nonempty set $J_{\alam} \subset [0,1]$ such that
$J_{\alam} = \theta_1(J_{\alam}) \cup \theta_2(J_{\alam})$.
Note that $J_0$ is the {\it middle thirds} Cantor set.

For $\alam \in [0,1]$ fixed, and $0 < s$, let
$X = C^2[0,1]$ and $Y = C[0,1]$, and define 
\begin{equation*}
g_1(x):= g_2(x):= g(x) := \theta_1^{\prime}(x)
= \frac{1}{3+2 \alam}(1 + \tfrac{7}{2}
\alam x^{5/2}).
\end{equation*}
As in Section~\ref{sec:intro}, define $\Lambda_s:X \to X$ and
$L_s:Y \to Y$ by the same formula:
\begin{equation}
\label{3.41} 
(\Lambda_s(w))(x) = g(x)^s[w(\theta_1(x)) + w(\theta_2(x))].
\end{equation}
Theorem~\ref{thm:1.1} implies that $r(L_s) = r(\Lambda_s)$; and it follows, for
example, from theorems in \cite{N-P-L} that the Hausdorff dimension of
$J_{\alam}$ is the unique value of $s$, $0 < s \le 1$, for which
$r(\Lambda_s) =1$.

If $w \in Y$ is a nonnegative function, we have that
\begin{equation*}
(L_s(w))(x) \ge \Big(\frac{1}{3+2 \alam}\Big)^s
[w(\theta_1(x)) + w(\theta_2(x))] \ge \Big(\frac{1}{5}\Big)^s
[w(\theta_1(x)) + w(\theta_2(x))] .
\end{equation*}
If $u(x) =1$ for $0 \le x \le 1$, it follows that
\begin{equation*}
L_s(u) \ge \Big(\frac{1}{5}\Big)^s (2u),
\end{equation*}
which implies that $r(L_s) \ge 2 (1/5)^s$.  If $\log$ denotes the natural
logarithm and $0 \le s < \log(2)/\log(5) \approx .4307$, it follows that
$r(L_s) >1$.  Thus if one is only interested in
$s$ with $r(L_s) \le 1$, one may restrict attention to $s \ge
\log(2)/\log(5)$.

In order to apply Theorem~\ref{thm:3.4}, we must
determine a range of $s >0$ such that
\begin{equation}
\label{3.42}
g^{\prime\prime}(x) g(x) - (1-s) [g^{\prime}(x)]^2 >0, \qquad 0 < x \le 1.
\end{equation}
The other hypotheses of Theorem~\ref{thm:3.4} can be
trivially verified. A calculation gives, for $0 < x \le 1$ that
\begin{equation*}
g^{\prime\prime}(x) g(x) - (1-s)[g^{\prime}(x)]^2 
= \frac{1}{(3+2a)^2}(\tfrac{7\alam }{2})(\tfrac{5}{2}) x^{1/2}
\Big[(\tfrac{3}{2}) 
 +(\tfrac{7 \alam}{2}) (-1 + \tfrac{5}{2}s) x^{5/2} \Big].
\end{equation*}
Assuming that $a >0$ and noting that $0 < u :=x^{5/2} \le 1$ if and only
if $0 < x \le 1$, we see that \eqref{3.42} is satisfied if and only if
\begin{equation*}
(\tfrac{3}{2})  +(\tfrac{7 \alam}{2}) (-1 + \tfrac{5}{2}s) u > 0 \qquad
\text{for} \qquad 0 \le u \le 1,
\end{equation*}
which is equivalent to
\begin{equation*}
(\tfrac{3}{2})  +(\tfrac{7 \alam}{2}) (-1 + \tfrac{5}{2}s) > 0.
\end{equation*}
The latter inequality is certainly satisfied for $0 < a \le 3/7$ and $s >0$;
and if $3/7 \le a \le 1$, the inequality is satisfied for $s > \tfrac{2}{5}
[1 - 3/(7a)]$.  Thus we conclude that for $0 \le a \le 1$, $0<x\le 1$, and $s
>0$, \eqref{3.42} is satisfied if and only if
\begin{equation}
\label{3.43}
s > \tfrac{2}{5}  [1 - 3/(7a)].
\end{equation}
It follows from Theorem~\ref{thm:3.4} that if $0 < a \le 1$, $s >0$ and
\eqref{3.43} is satisfied, then $v_s^{\prime}(x) >0$ and
$v_s^{\prime\prime}(x) >0$ for $0 < x \le 1$.

It remains to apply Theorems~\ref{thm:3.1} and \ref{thm:3.2} to our example.
We assume, in the notation of (H5.1) that $\mu =1$ and $m \ge 2$.  The
eigenfunction $v_s(\cdot)$ for \eqref{3.41} depends on the parameter $\alam$,
although this is not indicated in our notation, and, of course our
various constants depend on $\alam$.  Since 
$\theta_j^{\prime}(x) = g(x)$, the constant $\kappa = \kappa(\alam)$ in (H5.1)
is given by
\begin{equation}
\label{3.44}
\kappa(\alam) = \max \{g(x): 0 \le x \le 1\} = (2+7 \alam)/(6+ 4 \alam)
<1.
\end{equation}

The constant $C_1 = C_1(\alam)$ in Theorem~\ref{thm:3.1} is defined by
\begin{align*}
C_1 = C_1(\alam) &= \sup\{[(\tfrac{7 \alam}{2})(\tfrac{5}{2}) x^{3/2})]
[1 + (\tfrac{7\alam}{2}) x^{5/2})]^{-1}: 0 \le x \le 1\}
\\
&= \sup\{[(\tfrac{7 \alam}{2})(\tfrac{5}{2}) u^3]
[1 + (\tfrac{7 \alam}{2}) u^5)]^{-1}: 0 \le u \le 1\}.
\end{align*}
An elementary but tedious calculus argument, which we leave to the reader,
yields
\begin{equation}
\label{3.45}
C_1(\alam) = \begin{cases}
(\frac{7 \alam}{2})(\frac{5}{2})[1 + (\frac{7}{2}) \alam]^{-1}, 
& 0 < \alam \le \tfrac{3}{7}
\\
(\frac{7 \alam}{2})(\frac{3}{7 \alam})^{3/5}, & \tfrac{3}{7} 
\le \alam \le 1.
\end{cases}
\end{equation}
It follows from Theorems~\ref{thm:3.1}
and \eqref{1.8} that for $0 < x \le 1$,
\begin{equation*}
0 < \frac{v_s^{\prime}(x)}{v_s(x)} \le s C_1(\alam) [1- \kappa(\alam)]^{-1}
= s C_1(\alam) (6 + 4 \alam)/(4 - 3 \alam):= M_1(\alam).
\end{equation*}
An easy calculation also yields that
\begin{equation}
\label{3.47}
\max\{ \theta_1^{\prime\prime}(x) : 0 \le x \le 1\}
= (\tfrac{7 \alam}{2}) (\tfrac{5}{6 + 4 \alam}):= E_2(\alam).
\end{equation}

By definition (see \eqref{3.15} with $\mu =1$) we have that
\begin{align*}
C_2:= C_2(\alam) &= \sup\{g^{\prime\prime}(x)/g(x): 0 \le x \le 1\}
\\
&= \tfrac{15}{4} \sup\{(\tfrac{7 \alam}{2}) x^{1/2})
[1 + (\tfrac{7 \alam}{2}) x^{5/2})]^{-1}: 0 \le x \le 1\}
\\
&= (\tfrac{15}{4}) \sup\{(\tfrac{7\alam}{2})u 
[1 + (\tfrac{7\alam}{2})u^5]^{-1}: 0 \le u \le 1\},
\end{align*}
and a simple calculus exercise yields
\begin{equation}
\label{3.48}
C_2(\alam) = \begin{cases}
(\frac{15}{4})(\frac{7 \alam}{2})[1 + (\frac{7 \alam}{2})]^{-1}, 
& 0 < \alam \le \tfrac{1}{14}
\\
3 (\frac{1}{4})^{1/5} (\frac{7 \alam}{2})^{4/5}, & \tfrac{1}{14}
\le \alam \le 1.
\end{cases}
\end{equation}

Using \eqref{3.39} and \eqref{3.40}, we now find that for
$8/35 < s \le 1$, $0 < \alam \le 1$, and $0 < x \le 1$, we have
\begin{equation*}
0 < \frac{v_s^{\prime\prime}(x)}{v_s(x)} 
\\
\le  \Big[s G_2(\alam) +\frac{ 2 s^2 C_1(\alam)^2 \kappa(\alam)}
{1 - \kappa(\alam)} + \frac{ s C_1(\alam) E_2(\alam)}{1 - \kappa(\alam)}\Big]
\big[1 - \kappa(\alam)^2\big]^{-1},
\end{equation*}
where $\kappa(\alam)$, $C_1(\alam)$, $E_2(\alam)$, and $C_2(\alam)$ are given
by \eqref{3.44}, \eqref{3.45}, \eqref{3.47}, and \eqref{3.48}, respectively,
and $G_2(\alam)$ is given by
\begin{equation}
\label{3.50}
G_2(\alam) = \max_{0 \le x \le 1} \frac{g^{\prime\prime}(x) g(x) 
- (1-s) [g^{\prime}(x)]^2}{g(x)^2} < C_2(\alam).
\end{equation}
Equation \eqref{3.43} implies that $G_2(\alam) >0$ for $0 < \alam \le 1$
if we assume that $8/35 <s \le 1$.


\section{Estimates for derivatives of $v_s$: The Case of M\"obius
  Transformations}
\label{sec:mobius}

When the maps $\theta_b$, $b \in \B$ are M\"obius transformations, one
can obtain much sharper estimates for $\max\{D^k v_s(x)/v_s(x): x \in \bar
H\}$ than were available in Section~\ref{sec:1d-deriv}.

We shall be interested in the one dimensional case, and our maps will
eventually be of the form $\theta_b(x):= 1/(x+b)$, where $b >0$.  The special
case where $\B$ is a subset of the positive integers has been of great
interest because of connections with continued fractions.  See, for example,
\cite{Bourgain-Kontorovich}, \cite{Bumby1}, \cite{Bumby2}, \cite{Cusick1},
\cite{Cusick2}, \cite{R}, \cite{Good}, \cite{Hensley1}, \cite{Hensley3}, and
\cite{Hensley2}. However, for our immediate purposes, nothing is gained by
restricting to $\B \subset \N$.

\begin{lem}
\label{lem:6.1}
Let $\B$ denote a finite collection of complex numbers $b$ such that
$\Re(b) \ge \gamma >0$ for all $b \in \B$.  For $b \in \B$, define
$M_b =
\bigl(\begin{smallmatrix} 0 & 1 \\
  1 & b \end{smallmatrix}\bigr)$ and $\theta_{b}(z) = 1/(z+b)$ for
$\Re(z) \ge 0$.  Let $b_j$, $j \ge 1$, denote a sequence of elements of $\B$.
Then for $n \ge 1$, we have 
\begin{equation}
\label{6.0}
M_{b_1} M_{b_2} \cdots M_{b_n} = \begin{pmatrix} A_{n-1} & A_n \\ B_{n-1} & B_n
\end{pmatrix},
\end{equation}
and
\begin{equation*}
M_{b_n} M_{b_{n-1}} \cdots M_{b_1} = \begin{pmatrix} A_{n-1} & B_{n-1} \\ A_n & B_n
\end{pmatrix},
\end{equation*}
where $A_0 =0 $, $A_1 =1$, $B_0 =1$, $B_1 = b_1$ and for $n \ge 2$,
\begin{equation}
\label{6.1}
A_{n+1} = A_{n-1} + b_{n+1} A_n \text{ and }
B_{n+1} = B_{n-1} + b_{n+1} B_n.
\end{equation}
If $G:= \{z \in \C: \Re(z) \ge 0\}$ and $D_{\gamma^{-1}}: = \{z \in \C:
|z- \gamma^{-1}| \le \gamma^{-1}\}$, then for all $b \in \B$,
\begin{equation}
\label{6.2}
\theta_b(G) \subset D_{\gamma^{-1}}.
\end{equation}
Also, we have for all $z \in G$, $(\theta_{b_1} \circ \theta_{b_{2}} \circ
\cdots \circ \theta_{b_n})(z) \in D_{\gamma^{-1}}$ and
\begin{equation}
\label{6.3}
(\theta_{b_1} \circ \theta_{b_{2}} \circ \cdots \circ \theta_{b_n})(z)
= (A_{n-1} z + A_n)/(B_{n-1} z + B_n),
\end{equation}
and 
\begin{equation}
\label{6.4}
(\theta_{b_n} \circ \theta_{b_{n-1}} \circ \cdots \circ \theta_{b_1})(z)
= (A_{n-1} z + B_{n-1})/(A_n z + B_n).
\end{equation}
For all $n \ge 0$, $B_n \neq 0$ and $\Re(B_{n+1}/B_{n}) \ge \gamma$,
while for all $n \ge 1$, $A_n \neq 0$ and $\Re(A_{n+1}/A_{n}) \ge \gamma$.
For all $b, c \in \B$, $\theta_b \circ \theta_c |_G$ is a Lipschitz map
(with respect to the Euclidean norm on $\C$) and
\begin{equation}
\label{6.5}
|\theta_b(\theta_c(z)) - \theta_b(\theta_c(w))| \le \frac{1}{4 \gamma^2}
|z-w|, \qquad \forall z,w \in G.
\end{equation}
\end{lem}
\begin{proof}
\begin{equation*}
M_{b_1} =  \begin{pmatrix} 0 & 1 \\
  1 & b_1 \end{pmatrix}
= \begin{pmatrix} A_0 & A_1 \\
  B_0 & B_1 \end{pmatrix}.
\end{equation*}
We argue by induction and assume that \eqref{6.0} is satisfied for some $n
  \ge 1$. Then we obtain
\begin{multline*}
M_{b_1} M_{b_2} \cdots M_{b_n} M_{b_{n+1}} 
= \begin{pmatrix} A_{n-1} & A_n \\ B_{n-1} & B_n
\end{pmatrix} \begin{pmatrix} 0 & 1 \\ 1 & b_{n+1} \end{pmatrix}
\\
= \begin{pmatrix} A_n & A_{n-1} + b_{n+1} A_n \\
B_n & B_{n-1} + b_{n+1} B_n \end{pmatrix}
= \begin{pmatrix} A_n & A_{n+1}\\
B_n & B_{n+1} \end{pmatrix},
\end{multline*}
which completes the inductive proof.  The formula for 
$M_{b_n} M_{b_{n-1}} \cdots M_{b_1}$ follows by taking the transpose of the
formula for $M_{b_1} M_{b_2} \cdots M_{b_n}$.

Equations \eqref{6.3} and \eqref{6.4} are now standard results for
 M\"obius transformations.  If $z \in G$, $z + b \in \{w: \Re(w) \ge
 \gamma\}$. A standard exercise shows that the map $w \mapsto 1/w$ takes
the set $\{w: \Re(w) \ge  \gamma\}$ into $D_{\gamma^{-1}}$, and this
establishes \eqref{6.2}.

Notice that $B_0=1$ and $B_1 = b_1$ so $B_0$ and $B_1$ are nonzero and
$\Re(B_1/B_0) \ge \gamma$.  We argue by induction and assume that we have
proved $B_j \neq 0$ for $0 \le j \le n$ and $\Re(B_{j+1}/B_j) \ge \gamma$
for $0 \le j \le n-1$. We then obtain that
\begin{equation*}
\Re(B_{n+1}/B_n) = \Re(B_{n-1}/B_n) + \Re(b_{n+1}) \ge \Re(B_{n-1}/B_n) + \gamma.
\end{equation*}
Writing $\beta= B_n/B_{n-1}$, so $\Re(\beta) \ge \gamma$,
we see that
\begin{equation*}
\Re(B_{n-1}/B_n) = \Re(1/\beta) = \Re(\bar \beta/|\beta|^2) 
\ge \gamma/|\beta|^2,
\end{equation*}
so
\begin{equation*}
\Re(B_{n+1}/B_n)  = \gamma( 1 + |B_{n-1}/B_n|^2) > \gamma
\end{equation*}
and $B_{n+1} \neq 0$.

The proof that $A_n \neq 0$ for all $n \ge 1$ and $\Re(A_{n+1}/A_n) \ge
\gamma$ for all $n \ge 1$ follows by a similar induction argument and is
left to the reader.

Notice that $\det(M_{b_1} M_{b_2} \cdots M_{b_n}) = (-1)^n$, so
\begin{equation}
\label{6.6}
\frac{d}{dz}(\theta_{b_1} \circ \theta_{b_{2}} \circ \cdots \circ \theta_{b_n})(z)
= \frac{(-1)^n}{(B_{n-1} z + B_n)^2}
= \frac{(-1)^n}{B_{n-1}^2(z + B_n/B_{n-1})^2} .
\end{equation}
If we can prove that $|B_{n-1}^2(z + B_n/B_{n-1})^2| \ge L$ for all
$z \in G$, it will follow that for all $z, w \in G$,
\begin{equation}
\label{6.7}
|(\theta_{b_1} \circ \theta_{b_{2}} \circ \cdots \circ\theta_{b_n})(z)
- (\theta_{b_1} \circ \theta_{b_{2}} \circ \cdots \circ \theta_{b_n})(w)|
\le \frac{1}{L} |z-w|.
\end{equation}
However, for $n \ge 2$ and $z \in G$,
\begin{multline*}
|z + B_n/B_{n-1}| \ge \Re(z + B_n/B_{n-1}) \ge \Re(B_n/B_{n-1})
= \Re(B_{n-2}/B_{n-1} + b_n) 
\\
= \Re\Big(\frac{B_{n-1}}{B_{n-2}}\Big) \frac{|B_{n-2}|^2}
{|B_{n-1}|^2} + \Re(b_n)
\ge \gamma \frac{|B_{n-2}|^2}
{|B_{n-1}|^2} + \gamma.
\end{multline*}
This implies that
\begin{multline}
\label{6.8}
|B_{n-1}^2(z + B_n/B_{n-1})^2| \ge |B_{n-1}|^2 \gamma^2
(1 + |B_{n-2}|^2/|B_{n-1}|^2)^2 
\\
= \gamma^2 |B_{n-2}|^2 (|B_{n-1}|^2/|B_{n-2}|^2 + 2
+ |B_{n-2}|^2/|B_{n-1}|^2).
\end{multline}
Using \eqref{6.6} and \eqref{6.8}, we see that for $z \in G$ and $n \ge 2$,
\begin{equation*}
\Big|\frac{d}{dz}(\theta_{b_1} \circ \theta_{b_{2}} \circ \cdots \circ
\theta_{b_n})(z) \Big| \le (4 \gamma^2 |B_{n-2}|^2)^{-1},
\end{equation*}
with strict inequality unless $|B_{n-1}| = |B_{n-2}|$, and this implies
\eqref{6.7}, with $L:= 4 \gamma^2 |B_{n-2}|^2$.

If we take $n=2$, so $B_{n-2}=1$, we find that for any two elements
$b_1$ and $b_2$ in $\C$ and for all $z, w \in G$, we have
\begin{equation*}
|\theta_{b_1}(\theta_{b_2}(z)) - \theta_{b_1}(\theta_{b_2}(w))|
\le \frac{1}{4 \gamma^2} |z-w|.
\end{equation*}
Taking $b_1 = b$ and $b_2 =c$, we obtain \eqref{6.5}.
\end{proof}

For the remainder of this section we shall restrict ourselves to the case
in Lemma~\ref{lem:6.1} that $\B$ is a subset of the positive reals and
$b \ge \gamma >0$ for all $b \in \B$.
\begin{lem}
\label{lem:6.2}
Let $\B$ denote a finite set of positive reals such that $b \ge \gamma >0$ for
all $b \in \B$ and let notation be as in Lemma~\ref{lem:6.1}. If $b_j$,
$j \ge 1$ denotes a sequence of elements of $\B$, then for all $n \ge 0$,
$B_n >0$ and $B_{n+1}/B_n \ge \gamma$ and for all $n \ge 1$, $A_n >0$
and $A_{n+1}/A_n \ge \gamma$ and $B_n/A_n \ge \gamma$.  For all $k \ge 0$,
we have
\begin{equation}
\label{6.10}
B_{2k} \ge (1+ \gamma^2)^k \qquad \text{and} \qquad 
B_{2k+1} \ge \gamma (1+ \gamma^2)^k.
\end{equation}
For all $\w = (b_1, b_2, \cdots, b_{2m}) \in \B_{2m}$, $m \ge 1$, and
$z, w \in G$, we have
\begin{equation}
\label{6.11}
|(\theta_{b_1} \circ \theta_{b_{2}} \circ \cdots \circ \theta_{b_{2m}})(z)
-(\theta_{b_1} \circ \theta_{b_{2}} \circ \cdots \circ \theta_{b_{2m}})(w)| 
\le (1 +\gamma^2)^{-2m}|z-w|
\end{equation}
and
\begin{equation}
\label{6.12}
|\theta_{\w}(z) - \theta_{\w}(w)| \le  (1 +\gamma^2)^{-2m}|z-w|.
\end{equation}
\end{lem}
\begin{proof}
Using \eqref{6.1} it is an easy induction argument (left to the reader)
to prove that $A_n >0$ for all $n \ge 1$ and $B_n >0$ for all $n \ge 0$.
It then follows immediately from  Lemma~\ref{lem:6.1} that
$A_{n+1}/A_n \ge \gamma$ for $n \ge 1$ and $B_{n+1}/B_n \ge \gamma$ for $n \ge
0$.

Since $B_1 = b_1 \ge \gamma$ and $A_1 =1$, we see that $B_1/A_1 \ge \gamma$.
Arguing by induction, assume that we have proved that
$B_j/A_j \ge \gamma$ for $1 \le j \le n$.  Then we obtain
\begin{equation*}
\frac{B_{n+1}}{A_{n+1}} = \frac{B_{n-1} + b_{n+1} B_n}{A_{n-1} + b_{n+1} A_n}
\ge \frac{A_{n-1} \gamma  + b_{n+1} A_n \gamma}{A_{n-1} + b_{n+1} A_n}
= \gamma,
\end{equation*}
which completes the inductive argument.

We next claim that for all $k \ge 0$, the first inequality in \eqref{6.10}
holds.  For $k=0$, this is immediate, since $B_0=1$. We argue by induction and
assume that we have proved the first inequality in \eqref{6.10} for some $k
\ge 0$. We have that
\begin{equation*}
B_{2k+1} = B_{2k-1} + b_{2k+1} B_{2k} \ge B_{2k-1} + \gamma  B_{2k},
\end{equation*}
and this implies that
\begin{multline*}
B_{2k+2} = B_{2k} + b_{2k+2} B_{2k+1} \ge B_{2k} + \gamma  B_{2k+1}
\\
\ge B_{2k} + \gamma  B_{2k-1} + \gamma^2 B_{2k}
\ge (1+ \gamma^2) B_{2k} \ge (1+ \gamma^2)^{k+1}.
\end{multline*}
This completes the induction argument.

Since $B_1 = b_1 \ge \gamma$, and $B_{2k+1} = B_{2k-1} + b_{2k} B_{2k}
\ge \gamma(1+\gamma^2)^k$ for $k \ge 1$, we obtain the second part
of \eqref{6.10}.

For $z \in H$, we obtain from Lemma~\ref{lem:6.1} that
\begin{equation}
\label{6.14}
\Big|\frac{d}{dz}(\theta_{b_1} \circ \theta_{b_{2}} \circ
\cdots \circ \theta_{b_{2m}})(z)\Big|
\le |B_{2m-1}z + B_{2m}|^{-2}
\end{equation}
and
\begin{equation}
\label{6.15}
\Big|\frac{d}{dz} \theta_{\w}(z)\Big|:=
\Big|\frac{d}{dz}(\theta_{b_{2m}} \circ \theta_{b_{2m-1}} \circ \cdots 
\circ \theta_{b_{1}})(z)\Big| \le |A_{2m}z + B_{2m}|^{-2}.
\end{equation}
Because $B_{2m-1}$, $A_{2m}$, and $B_{2m}$ are positive,
$\Re(B_{2m-1}z + B_{2m}) \ge \Re(B_{2m}) \ge (1+ \gamma^2)^m$ and
$\Re(A_{2m}z + B_{2m}) \ge \Re(B_{2m}) \ge (1+ \gamma^2)^m$. This implies
that for all $z \in H$,
\begin{equation}
\label{6.16}
|B_{2m-1}z + B_{2m}|^{-2} \le (1+ \gamma^2)^{-2m} \quad \text{and} \quad
|A_{2m}z + B_{2m}|^{-2} \le (1+ \gamma^2)^{-2m}.
\end{equation}
Using \eqref{6.14}, \eqref{6.15}, and \eqref{6.16}, we obtain
\eqref{6.11} and \eqref{6.12}.
\end{proof}

\begin{remark}
\label{rem:6.1} 
Given $\w = (b_1, b_2, \cdots, b_n) \in \B_n$, we have defined $\theta_{\w} =
\theta_{b_n} \circ \theta_{b_{n-1}} \circ \cdots \circ \theta_{b_1}$ (to
conform to notation used in \cite{N-P-L}). However, we could also have
defined $\tilde \theta_{\w} = \theta_{b_1} \circ \theta_{b_{2}} \circ \cdots
\circ \theta_{b_n}$, which is perhaps more natural.  Similarly, we have
defined $g_{\w}(z)$ by
\begin{equation*}
g_{b_n}(\theta_{b_{n-1}} \circ \theta_{b_{n-2}} \circ \cdots
\circ \theta_{b_1}(z))
g_{b_{n-1}}(\theta_{b_{n-2}} \circ \theta_{b_{n-3}} \circ \cdots 
\circ \theta_{b_1}(z))\cdots g_{b_2}(\theta_{b_1}(z)) g_{b_{1}}(z).
\end{equation*}
However, we could have defined
\begin{equation*}
\tilde g_{\w}(z) = g_{b_1}(\theta_{b_{2}} \circ \theta_{b_{3}} \circ \cdots
\circ \theta_{b_n}(z))
g_{b_{2}}(\theta_{b_{3}} \circ \theta_{b_{4}} \circ \cdots 
\circ \theta_{b_n}(z)) \cdots g_{b_{n-1}}(\theta_{b_n}(z)) g_{b_{n}}(z).
\end{equation*}
We leave to the reader the verification that
\begin{equation*}
(\Lambda_s^n f)(z) = \sum_{\w \in \B_n}[g_{\w}(z)]^s f(\theta_{\w}(z))
=  \sum_{\w \in \B_n}[\tilde g_{\w}(z)]^s f(\tilde \theta_{\w}(z)).
\end{equation*}
\end{remark}

\begin{thm}
\label{thm:6.3}
Let $\B$ be a finite set of positive reals such that  $b \ge \gamma >0$ for
all $b \in \B$. For such $b$ and all $x \ge 0$, define
$\theta_b(x) = (x+b)^{-1}$.  If $A \ge \gamma^{-1}$, define
$H = \{x \in \R: 0 <x <A\}$, so $\theta_b(\bar H) \subset [0, \gamma^{-1}]$.
Assume that $m$ is a positive integer
and $g_b : [0,A] \to \R$ is a $C^m$ function such that $g_b(x) >0$
for all $x \in [0,A]$.  Let $X=X_m$ denote the Banach space $C^m(\bar H)$
and for $s >0$ define
\begin{equation*}
(\Lambda_s f)(x) = \sum_{b \in \B} [g_b(x)]^s f(\theta_b(x)).
\end{equation*}
Then all the hypotheses of Theorem~\ref{thm:1.1} are satisfied, so
$\Lambda_s$ has a unique (to within normalization) strictly positive
eigenfunction $v_s \in X$ with eigenvalue $r(\Lambda_s) >0$.
Furthermore, in our usual notation, for $1 \le j \le m$ and $x \in [0,A]$,
\begin{equation}
\label{6.18}
\frac{D^j v_s(x)}{v_s(x)} = \lim_{n \rightarrow \infty}
\frac{\sum_{\w \in \B_n} D^j g_{\w}(x)}{\sum_{\w \in \B_n} g_{\w}(x)}.
\end{equation}
\end{thm}
\begin{proof}
Theorem~\ref{thm:6.3} follows from Theorem~\ref{thm:1.1} and
Remark~\ref{rem:1.2} once we verify that conditions (H4.1), (H4.2), and
(H4.3) in Section~\ref{sec:exist} are satisfied.  Conditions (H4.1) and (H4.2)
are obviously satisfied. Also, it follows from \eqref{6.11} or \eqref{6.12}
in Lemma~\ref{lem:6.2} that for all $x,y \in [0,A]$ and all $b_1, b_2 \in \B$,
\begin{equation*}
|\theta_{b_1}(\theta_{b_2}(x)) - \theta_{b_1}(\theta_{b_2}(y))|
\le (1+ \gamma^2)^{-2} |x-y|,
\end{equation*}
which verifies (H4.3) with $\mu =2$ and $\kappa = (1+ \gamma^2)^{-2}$.
\end{proof}

Notice that if $g_b(\cdot)$ is $C^{\infty}$ on $[0,A]$, Theorem~\ref{thm:6.3}
implies that $v_s(\cdot)$ is $C^{\infty}$ on $[0,A]$ and \eqref{6.18}
holds for all $j \ge 1$.

We are interested in Theorem~\ref{thm:6.3} in the special case that
$g_b(x) = |\theta_b^{\prime}(x)|^s = (x+b)^{-2s}$.  In this case, it is
easy to verify that for $\mu \ge 1$,
\begin{equation*}
(\Lambda_s^{\mu} f)(x) = \sum_{\w \in \B_{\mu}} |\theta_{\w}^{\prime}(x)|^s
f(\theta_{\w}(x)).
\end{equation*}
If $\w = (b_1, b_2, \cdots, b_{\mu}) \in \B_{\mu}$ and $A_j$ and $B_j$ are
as defined in Lemma~\ref{lem:6.1}, recall that
\begin{equation*}
[g_{\w}(x)]^s = |\theta_{\w}^{\prime}(x)|^s
= (A_{\mu} x + B_{\mu})^{-2s} = A_{\mu}^{-2s}(x + B_{\mu}/A_{\mu})^{-2s}.
\end{equation*}
If $1 \le j \le m$, it follows that
\begin{multline}
\label{6.21}
\frac{D^j [g_{\w}(x)]^s }{[g_{\w}(x)]^s} = 
\frac{D^j [(x + B_{\mu}/A_{\mu})^{-2s}]}{(x + B_{\mu}/A_{\mu})^{-2s}}
\\
= (-1)^j(2s)(2s+1) \cdots (2s+j-1) (x + B_{\mu}/A_{\mu})^{-j}.
\end{multline}

Lemma~\ref{lem:6.2} implies that $B_{\mu}/A_{\mu} \ge \gamma$ for all $\mu \ge
1$. On the other hand, if $\Gamma = \max \{b: b\in \B\}$, a calculation
gives
\begin{equation*}
B_1/A_1 = b_1 \le \Gamma \qquad \text{and} \qquad
B_2/A_2 = b_1 + b_2^{-1} \le \Gamma + \gamma^{-1}.
\end{equation*}
Let $K = \Gamma + \gamma^{-1}$ and, arguing inductively, assume that we
have proved, for some $n \ge 2$, that
\begin{equation}
\label{6.22}
B_j/A_j \le K, \quad 1 \le j \le n.
\end{equation}
Then we obtain
\begin{equation*}
\frac{B_{n+1}}{A_{n+1}} = \frac{B_{n-1} + b_{n+1} B_n}{A_{n-1} + b_{n+1} A_n}
\le  \frac{K A_{n-1} + K b_{n+1} A_n}{A_{n-1} + b_{n+1} A_n} = K,
\end{equation*}
which proves that \eqref{6.22} holds for all $n$.  It follows that
for $0 \le x \le A$ and $\mu \ge 1$, we have
\begin{equation}
\label{6.23}
(K + A)^{-j} \le (x + B_{\mu}/A_{\mu})^{-j} \le \gamma^{-j}.
\end{equation}
Using \eqref{6.23} in \eqref{6.21}, we obtain for $0 \le x \le A$ 
and $\mu \ge 1$,
\begin{multline}
\label{6.24}
(2s)(2s+1) \cdots (2s+j-1)(K + A)^{-j} \le
(-1)^j \frac{D^j [g_{\w}(x)] }{g_{\w}(x)}
\\
\le (2s)(2s+1) \cdots (2s+j-1)\gamma^{-j}.
\end{multline}

Thus we have proved the following corollary of Theorem~\ref{thm:6.3}.
\begin{cor}
\label{cor:6.4}
Let $\B$ be a finite set of positive real numbers and define
$\gamma = \min\{b : b \in \B\}$, $\Gamma = \max\{b : b \in \B\}$,
and $K = \gamma^{-1} + \Gamma$.  Let $A$ be any real number with
$A \ge \gamma^{-1}$ and for any positive integer $m$, define
$X = X_m = C^m([0,A])$. For $s >0$ define a bounded linear operator
$\Lambda_s: X_m \to X_m$ by
\begin{equation*}
(\Lambda_s f)(x) = \sum_{b \in \B} (x+b)^{-2s} f(\theta_b(x)),
\end{equation*}
where $\theta_b(x) = (x+b)^{-1}$.  Then $\Lambda_s$
has a unique (to within normalization) strictly positive eigenfunction $v_s
\in X_m$ and $v_s$ is actually infinitely differentiable.  Furthermore, for
integers
$j \ge 1$, we have the estimates 
\begin{multline}
\label{6.25}
(2s)(2s+1) \cdots (2s+j-1)(K + A)^{-j} \le
(-1)^j \frac{D^j [v_s(x)] }{v_s(x)}
\\
\le (2s)(2s+1) \cdots (2s+j-1)\gamma^{-j}, \qquad x \in [0,A].
\end{multline}
\end{cor}
\begin{proof}
Equation \eqref{6.25} follows from \eqref{6.18} and \eqref{6.24} by letting
$n \rightarrow \infty$, where $\w \in \B_{n}$.
\end{proof}

\begin{remark}
\label{rem:6.2}
Suppose that assumptions and notation are as in Corollary~\ref{cor:6.4},
so $v_s: [0,A] \mapsto \R$ is strictly positive and $v_s \in C^m([0,A])$.
Then $v_s(\cdot)$ has an analytic, complex-valued extension to
$H = \{z \in \C: \Re(z) >0\}$.  The idea of the proof is to consider
the linear operator
\begin{equation*}
(R_sf)(z) = \sum_{b \in \B} (z+b)^{-2s} f([z+b]^{-1}),
\end{equation*}
where $f$ is an element of an appropriate Banach space of complex
analytic functions $f(\cdot)$ defined on 
$\{z \in \C : |z- A/2| < A/2\} :=D$ and continuous on $\bar D$.

Since we shall not use this analyticity result, we omit the proof, but its
interest for us is precisely that in more general situations, it does not
seem possible to study our problem in a Banach space of analytic functions.
Suppose that $\B$ is a finite set of complex numbers as in Lemma~\ref{lem:6.1}
and $\theta_b(z) = (z+b)^{-1}$ for $b \in \B$ and $\Re(z) \ge 0$. If $A >
\gamma^{-1}$ and $D$ is as above, one can prove that
$\{\theta_b(z) : z \in \bar D, b \in 
\B\}$ is contained in a compact subset of $D$. 
For $m \ge 2$ and $s >0$, one defines $\Lambda_s: C^m(\bar D) \to
C^m(\bar D)$ by
\begin{equation*}
(\Lambda_sf)(z) = \sum_{b \in \B} |z+b|^{-2s} f(\theta_b(z)),
\end{equation*}
(note $(z+b)^{-2s}$ has been replaced by $|z+b|^{-2s}$), and $\Lambda_s$ has a
unique, normalized eigenfunction $v_s(\cdot)$ such that $v_s(z) >0$ for all $z
\in \bar D$. The eigenvalue of $v_s$ is $r(\Lambda_s)$, the spectral radius of
$\Lambda_s$.  In the context of {\it complex continued fractions} (see
\cite{Gardner-Mauldin}, \cite{Mauldin-Urbanski}, \cite{N-P-L}, and
\cite{Priyadarshi}), one wants to estimate $r(\Lambda_s)$.  However $z \mapsto
|z+b|^{-2s}$ and $z \mapsto v_s(z)$ are $C^{\infty}$, but not complex analytic
on $D$.  If $\B$ is not contained in $\R$, in general there does not seem to
be a natural bounded linear operator in a Banach space of analytic functions
with spectral radius $r(\Lambda_s)$.  In this generality, the linear operator
$R_s$ can still be defined in a Banach space of analytic functions, but will
almost always have spectral radius less than $r(\Lambda_s)$.

\end{remark}

\section{Computing the Spectral Radius of $A_s$ and $B_s$}
\label{sec:compute-sr}
In previous sections, we have constructed matrices $A_s$ and $B_s$ such that
$r(A_s) \le r(L_s) \le r(B_s)$.  The $(n+1) \times (n+1)$ matrices $A_s$ and
$B_s$ have nonnegative entries, so the Perron-Frobenius theory for such
matrices implies that $r(B_s)$ is an eigenvalue of $B_s$ with corresponding
nonnegative eigenvector, with a similar statement for $A_s$. One might also
hope that standard theory (see \cite{D}) would imply that $r(B_s)$,
respectively $r(A_s)$, is an eigenvalue of $B_s$ with algebraic multiplicity
one and that all other eigenvalues $z$ of $B_s$ (respectively, of $A_s$)
satisfy $|z| < r(B_s)$ (respectively, $|z| < r(A_s)$). Indeed, this would be
true if $B_s$ were {\it primitive}, i.e., if $B_s^k$ had all positive entries
for some integer $k$.  However, typically $B_s$ has many zero columns and
$B_s$ is neither primitive nor {\it irreducible} (see \cite{D}); and the same
problem occurs for $A_s$.  Nevertheless, the desirable spectral properties
mentioned above are satisfied for both $A_s$ and $B_s$. Furthermore $B_s$ has
an eigenvector $\wrm_s$ with all positive entries and with eigenvalue
$r(B_s)$; and if $x$ is any $(n+1) \times 1$ vector with all positive entries,
\begin{equation*}
\lim_{k \rightarrow \infty} \frac{B_s^k(x)}{\|B_s^k(x)\|} =
  \frac{\wrm_s}{\|\wrm_s\|},
\end{equation*}
where the convergence rate is geometric.  Of course, corresponding theorems
hold for $A_s$.  Such results justify standard numerical algorithms for
approximating $r(B_s)$ and $r(A_s)$.

In this section, we shall prove these assertions.  The basic point is
simple. Although $A_s$ and $B_s$ both map the cone $K$ of nonnegative vectors
in $\R^{n+1}$ into itself, $K$ is {\it not} the natural cone in which such
matrices should be studied.

To outline our method of proof, it is convenient to describe, at least in the
finite dimensional case, some classical theorems concerning linear maps $L:
\R^N \to \R^N$ which leave a cone $\CalC \subset \R^N$ invariant.  Recall that
a closed subset $\CalC$ of $\R^N$ is called a closed cone if (i) $ax + by \in
\CalC$ whenever $a \ge 0$, $b \ge 0$, $x \in \CalC$ and $y \in \CalC$ and (ii)
if $x \in \CalC \setminus \{0\}$, then $-x \notin \CalC$.  If $\CalC$ is a
closed cone, $\CalC$ induces a partial ordering on $\R^N$ denoted by
$\le_{\CalC}$ (or simply $\le$, if $\CalC$ is obvious) by $u \le_{\CalC} v$ if
and only if $v-u \in \CalC$.  If $u,v \in \CalC$, we shall say that $u$ and
$v$ are {\it comparable} (with respect to $\CalC$) and we shall write $u
\sim_{\CalC} v$ if there exist positive scalars $a$ and $b$ such that $v
\le_{\CalC} au$ and $u \le_{\CalC} b v$. {\it Comparable with respect to}
$\CalC$ partitions $\CalC$ into equivalence classes of comparable elements.
We shall henceforth assume that $\interior(\CalC)$, the interior of $\CalC$,
is nonempty.  Then an easy argument shows that all elements of
$\interior(\CalC)$ are comparable.  Generally, if $x_0 \in \CalC$ and
$\CalC_{x_0}: = \{x \in \CalC : x \sim_{\CalC} x_0\}$, all elements of
$\CalC_{x_0}$ are comparable.

Following standard notation, if $u,v \in \CalC$ are comparable elements, we
define \begin{align*}
M(u/v;\CalC) &= \inf\{\beta >0 : u \le \beta v\},
\\
m(u/v;\CalC) &= M(v/u;\CalC)^{-1} = \sup\{\alpha >0 : \alpha v \le u\}.
\end{align*}
If $u$ and $v$ are comparable elements of $\CalC \setminus \{0\}$, we define
Hilbert's projective metric $d(u,v;\CalC)$ by
\begin{equation*}
d(u,v;\CalC) = \log(M(u/v;\CalC)) + \log M(v/u;\CalC)).
\end{equation*}
We make the convention that $d(0,0;\CalC) =0$. If $x_0 \in \CalC \setminus
\{0\}$, then for all $u,v,w \in \CalC_{x_0}$, one can prove that (i)
$d(u,v;\CalC) \ge 0$, (ii) $d(u,v;\CalC) = d(v,u;\CalC)$, and (iii)
$d(u,v;\CalC) + d(v,w;\CalC) \ge d(u,w;\CalC)$. Thus $d$ restricted to
$\CalC_{x_0}$ is almost a metric, but $d(u,v;\CalC) =0$ if and only if $v =
tu$ for some $t >0$ and generally, $d(su,tv;\CalC) = d(u,v;\CalC)$ for all
$u,v \in \CalC_{x_0}$ and all $s >0$ and $t >0$.  If $\|\cdot\|$ is any norm
on $\R^N$ and $S:= \{ u \in \interior(\CalC): \|u\|=1\}$ (or, more generally,
if $x_0 \in \CalC \setminus \{0\}$ and $S = \{x \in \CalC_{x_0} : \|x\| =1\}$,
then $d(\cdot, \cdot; \CalC)$, restricted to $S \times S$, gives a metric on
$S$; and it is known that $S$ is a complete metric space with this metric.

With these preliminaries we can describe a special case of the Birkhoff-Hopf
theorem. We refer to \cite{S}, \cite{V}, and \cite{AA} for the original papers
and to \cite{U} and \cite{T} for an exposition of a general version of this
theorem and further references to the literature.  We remark that
P.~P.~Zabreiko, M.~A~Krasnosel$'$skij, Y.~V.~Pokornyi, and
A.~V.~Sobolev independently obtained closely related theorems; and we refer to
\cite{Y} for details.  If $\CalC$ is a closed cone as above, $S = \{x \in
\interior(\CalC): \|x\|=1 \}$, and $L: \R^N \to \R^N$ is a linear map such
that $L(\interior(\CalC)) \subset \interior(\CalC)$, we define
$\Delta(L;\CalC)$, {\it the projective diameter} of $L$ by
\begin{multline*}
\Delta(L;\CalC) = \sup\{d(Lx,Ly;\CalC): x,y \in \CalC \text{ and }
 Lx \sim_{\CalC} Ly\}
\\
= \sup\{d(Lx,Ly;\CalC): x,y \in \interior(\CalC)\}.
\end{multline*}
The Birkhoff-Hopf theorem implies that if $\Delta:= \Delta(L;\CalC) < \infty$,
then $L$ is a contraction mapping with respect to Hilbert's projective metric.
More precisely, if we define $\lambda = \tanh(\tfrac{1}{4} \Delta) <1$, then
for all $x,y \in \CalC \setminus \{0\}$ such that $x \sim_{\CalC} y$, we have
\begin{equation*}
d(Lx,Ly;\CalC) \le \lambda d(x,y;\CalC),
\end{equation*}
and the constant $\lambda$ is optimal.

If we define $\Phi:S \to S$ by $\Phi(x) = L(x)/\|L(x)\|$, it follows that
$\Phi$ is a contraction mapping with a unique fixed point $v \in S$, and
$v$ is necessarily an eigenvector of $L$ with eigenvector $r(L) := r=$
the spectral radius of $L$.  Furthermore, given any $x \in \interior(\CalC)$,
there are explicitly computable constants $M$ and $c <1$ (see Theorem 2.1
in \cite{U}) such that for all $k \ge 1$,
\begin{equation*}
\|L^k(x)/\|L^k(x)\| -v\| \le Mc^k;
\end{equation*}
and the latter inequality is exactly the sort of result we need.  Furthermore,
it is proved in Theorem 2.3 of \cite{U} that $r=r(L)$ is an
algebraically simple eigenvalue of $L$ and that if $\sigma(L)$ denotes
the spectrum of $L$ and $q(L)$ denotes the {\it spectral clearance of } $L$,
\begin{equation*}
q(L):= \sup\{|z|/r(L): z \in \sigma(L), z \neq r(L)\},
\end{equation*}
then $q(L) <1$ and $q(L)$ can be explicitly estimated.

If $A_s$, $B_s$, and $L_s$ are as in Section~\ref{sec:1dexps}, it remains to
find a suitable cone as above.  For the remainder of this section, $[a,b]$
will denote a fixed, closed bounded interval and $s$ a fixed nonnegative real.
For a given positive integer $n \ge 2$ and for integers $j$, $0 \le j \le n$,
we shall write $h = (b-a)/n$ and $x_j = a + jh$. $C$ will denote a fixed
constant and we shall always assume at least that
\begin{equation}
\label{9.1}
C h /4 \le 1.
\end{equation}
In our applications, $C$ will depend on $s$, but we shall not indicate this
dependence in our notation.  If $\wrm: \{x_j \, | \, 0 \le j \le n\} \to \R$,
one can extend $\wrm$ to a piecewise linear map $w^I:[a,b] \to \R$ by defining
\begin{equation}
\label{9.2}
w^I(x) = \frac{x-x_j}{h}\wrm_{j+1} + \frac{x_{j+1} - x}{h} \wrm_j, \quad
\text{for } x_j \le x \le x_{j+1}, \quad 0 \le j < n,
\end{equation}
where we have written $\wrm_j = \wrm(x_j)$.

We shall denote by $X_n$ (or $X$, if $n$ is obvious), the real vector space of
maps $\wrm:\{x_j \, | \, 0 \le j \le n\} \to \R$; obviously $X_n$ is linearly
isomorphic to $\R^{n+1}$, and we shall consider $A_s$, $B_s$, and $L_s$ as
maps of $X_n$ to $X_n$.  Note that in applying the results described above, we
set $N = n+1$.  For a given real $M >0$, we shall denote by $K_M \subset X_n$
the closed cone with nonempty interior given by
\begin{multline}
\label{9.3}
K_M = \{\wrm \in X_n \, | \, \wrm_{j+1} \le \wrm_j \exp(Mh) 
\\
\text { and } \wrm_j \le \wrm_{j+1}\exp(Mh), \quad 0 \le j < n\}.
\end{multline}
The reader can verify that if 
$\wrm =(\wrm_0, \wrm_1, \cdots, \wrm_n) \in K_M\setminus\{0\}$, 
then $\wrm_j >0$ for $0 \le j \le n$.

If $K_M \subset X_n$ are as above, suppose that $L:X_n \to X_n$ is a linear
map and that there exists $M^{\prime}$, $0 < M^{\prime} < M$, such that
$L(K_M\setminus\{0\}) \subset K_{M^{\prime}}\setminus\{0\}$. After correcting
the typo in the formula for $d_2(f,g)$ on page 286 of \cite{F}, it follows
from Lemma 2.12 on page 284 of \cite{F} that
\begin{equation*}
\sup\{d(f,g;K_M) : f,g \in K_{M^{\prime}}\setminus\{0\}\} 
\le 2 \log\Big(\frac{M + M^{\prime}}{M- M^{\prime}}\Big) 
+ 2 M^{\prime}(b-a) < \infty.
\end{equation*}
This implies that $\Delta(L;K_M) < \infty$, which in turn implies that
$L$ has a normalized eigenvector $v \in K_{M^{\prime}}$ with positive
eigenvalue $r = r(L) =$ the spectral radius of $L$.  
Furthermore, $r$ has algebraic multiplicity 1, $q(L) <1$, and
$\underset{k \to \infty}{\lim} \|L^k(x)/\|L^k(x)\| -v\| =0$ for all $x \in K_M
\setminus\{0\}$.  Thus it suffices to prove for appropriate maps $L$
that $L(K_M\setminus\{0\}) \subset K_{M^{\prime}}\setminus\{0\}$ for
some $M^{\prime} < M$.

If $x_j$, $0 \le j \le n$ are as above, define a map $Q:[a,b] \to [0,h^2/4]$
by
\begin{equation*}
Q(u) = (x_{j+1} - u)(u-x_j), \quad \text{for } x_j \le u \le x_{j+1}, \quad
0 \le j < n.
\end{equation*}

\begin{lem}
\label{lem:9.1}  Assume that $\beta \in K_{M_0} \setminus \{0\}$ for some
$M_0 >0$, that $0 < h \le 1$ and that $h$ and $C$ satisfy \eqref{9.1}.
Let $\theta:[a,b] \to [a,b]$ and define $\hat \beta_s \in X_n$ by
\begin{equation*}
\hat \beta_s(x_k) = [1 + \tfrac{1}{2} C Q( \theta(x_k))] [\beta(x_k)]^s.
\end{equation*}
Then $\hat \beta_s \in K_{M_1}$, where $M_1 = sM_0 + (1+h)/2 \le M_0 + 1$.
\end{lem}
\begin{proof}
Define $\psi \in X_n$ by
\begin{equation*}
\psi(x_k) = 1 + \tfrac{1}{2} C Q(\theta(x_k))
\end{equation*}
and suppose we can prove that $\psi \in K_{(1+h)/2}$.  For notational
convenience define $b(x_k) = [\beta(x_k)]^s$.
Then for $0 \le k <n$,
we obtain
\begin{multline*}
\psi(x_k) b(x_k) \le \psi(x_{k+1}) \exp([1+h]h/2) b(x_{k+1}) \exp(sM_0 h)
 \\
= \psi(x_{k+1}) b(x_{k+1}) \exp(M_1h),
\end{multline*}
and the same calculation gives
\begin{equation*}
\psi(x_{k+1}) b(x_{k+1}) \le \exp(M_1 h) \psi(x_k) b(x_k),
\end{equation*}
which implies that $x_k \mapsto \psi(x_k) b(x_k)$ is an element
of $K_{M_1}$.

Define $\delta = (1+h)/2$. Since $\psi(x_k) >0$ for $0 \le k \le n$, one
can check that $\psi(\cdot) \in K_{\delta}$ if and only if, for $0 \le k < n$,
\begin{equation*}
|\log( \psi(x_{k+1})) - \log( \psi(x_{k}))| = 
\Big| \log \Big(\frac{\psi(x_{k+1})}{\psi(x_{k})}\Big)\Big| \le \delta h.
\end{equation*}
Given $x_k$ and $x_{k+1}$ with $0 \le k <n$, write $\xi = \theta(x_k)$
and $\eta = \theta(x_{k+1})$.  Define $u:= \tfrac{1}{2} C Q(\theta(x_k))$ and
$v = \tfrac{1}{2} CQ(\theta(x_{k+1}))$, so $\psi(x_k) = 1 +u$
and $\psi(x_{k+1}) = 1 + v$.  Because $u$ and $v$ both lie in the
interval $[0, Ch^2/8]$, \eqref{9.1} implies that $|u-v| \le h/2$,
$|u| \le h/2$ and $|v| \le h/2$.  It follows
that
\begin{equation*}
|\log( \psi(x_{k})) - \log( \psi(x_{k+1}))| = |\log(1+u) - \log(1+v)|
= \Big|\int_{1+v}^{1+u} (1/t) \, dt \Big|.
\end{equation*}
Because $0 \le 1/t \le 1/(1 - h/2) \le 1 + h$ for all $t \in [1+v,1+u]$,
we obtain
\begin{equation*}
|\log( \psi(x_{k})) - \log( \psi(x_{k+1}))| \le (1+h)|u-v| \le (1+h) h/2,
\end{equation*}
which proves the lemma.
\end{proof}

\begin{lem}
\label{lem:9.2}
Let assumptions and notation be as in Lemma~\ref{lem:9.1}.  Let $\delta$
denote a fixed positive real and $s$ a fixed nonnegative real.  Assume,
in addition that $\theta:[a,b] \to [a,b]$ is a Lipschitz map with
$\Lip(\theta) \le c <1$ and that, for $h = (b-a)/n$ and $M_1$ as in
 Lemma~\ref{lem:9.1}, $\exp(-[M_1 + \delta] h)
\ge (1+c)/2$ and $M >0$ is such that $\exp(Mh) \ge 2$.  Define a linear map
$L_s:X_n \to X_n$ by 
\begin{equation*}
L_s(\wrm)(x_k) := w^I(\theta(x_k)) \hat \beta_s(x_k), \quad 0 \le k \le n.
\end{equation*}
Then, if $K_M \subset X_n$ is defined by \eqref{9.3}, $L_s(K_M) \subset
K_{M-\delta}$.
\end{lem}
\begin{proof}
For a fixed $k$, $0 \le k <n$, recall we have defined  $\xi = \theta(x_k)$
and $\eta = \theta(x_{k+1})$.  We must prove that if $h$ and $M$ satisfy the
above constraints and $\wrm \in K_M$, then
\begin{align*}
w^I(\xi) \hat \beta_s(x_k) &\le \exp([M-\delta]h) w^I(\eta) \hat
\beta_s(x_{k+1}),
\\
w^I(\eta) \hat \beta_s(x_{k+1}) &\le \exp([M-\delta]h) w^I(\xi) \hat
\beta_s(x_{k}).
\end{align*}
Using Lemma~\ref{lem:9.1}, we see that $x_k \mapsto \hat \beta_s(x_{k})$
is an element of $K_{M_1}$, so the above inequalities will be satisfied if
\begin{align}
\label{9.5}
w^I(\xi) &\le \exp([M-M_1-\delta]h) w^I(\eta),
\\
\label{9.6}
w^I(\eta) &\le \exp([M- M_1-\delta]h) w^I(\xi).
\end{align}
For notational convenience, we write $M_2 = M_1 + \delta$. By interchanging
the roles of $\xi$ and $\eta$, we can assume that $\eta \le \xi$, and it
suffices to prove that \eqref{9.5} and \eqref{9.6} are satisfied for $M$
and $h$ as in the statement of the Lemma.  Define $j = n-1$ if $\xi \ge
x_{n-1}$ and otherwise define $j$ to be the unique integer, $0 \le j < n-1$,
such that $x_j \le \xi < x_{j+1}$.  Because $0  \le \xi- \eta \le c h <h$,
there are only two cases to consider: either (i) $x_j \le \eta \le \xi$ 
or (ii) $x_{j-1} < \eta <x_j$ and $x_j \le \xi < x_{j+1}$.

We first assume that we are in case (i), so $\xi, \eta \in [x_j,x_{j+1}]$
and $0 \le \xi - \eta \le c h$,  Using \eqref{9.2}, we see that \eqref{9.5}
is equivalent to proving
\begin{multline}
\label{9.5p}
(x_{j+1} - \xi) \wrm_j + (\xi - x_j)  \wrm_{j+1} 
\\
\le
\exp([M-M_2]h)[(x_{j+1} - \eta) \wrm_j + (\eta - x_j)\wrm_{j+1}].
\end{multline}
Subtracting $(x_{j+1} - \eta) \wrm_j + (\eta - x_j)\wrm_{j+1}$ from both
sides of \eqref{9.5p} shows that \eqref{9.5p} will be satisfied if
\begin{multline}
\label{9.5pp}
(\xi - \eta) [\wrm_{j+1} -\wrm_j] 
\\
\le [\exp([M-M_2]h) -1][(x_{j+1} - \eta) \wrm_j + (\eta - x_j)\wrm_{j+1}].
\end{multline}
Equation \eqref{9.5pp} will certainly be satisfied if 
$\wrm_{j+1} \le \wrm_j$,
so we can assume that $\wrm_{j+1} - \wrm_j >0$ and 
$1 < \wrm_{j+1}/\wrm_j \le \exp(Mh)$.
If we divide both sides of \eqref{9.5pp} by $\wrm_j$ and recall that
$\xi-\eta \le ch$, we see that the left hand side of \eqref{9.5pp} is
dominated by $ch[\exp(Mh) -1]$, while the right hand side of \eqref{9.5pp}
is $\ge [\exp([M-M_2]h) -1]h$,  Thus, \eqref{9.5pp} will be satisfied if
\begin{equation}
\label{9.7}
c \le \frac{\exp([M-M_2]h) -1}{\exp(Mh) -1} = \exp(-M_2 h)
+ \frac{\exp(-M_2 h) -1}{\exp(Mh) -1}.
\end{equation}
If $h >0$ is chosen so that $\exp(-M_2 h) \ge (1+c)/2$, a calculation
shows that \eqref{9.7} will be satisfied if
$M \ge \log(2)/h$,
where $\log$ denotes the natural logarithm.  Thus, if $h >0$ satisfies
\eqref{9.1}, $M \ge \log(2)/h$, and $\exp(-M_2 h) \ge (1+c)/2$, 
\eqref{9.5} is satisfied in case (i).  Under the same conditions on
$h$ and $M$, an exactly analogous argument shows that (in case (i)),
\eqref{9.6} is also satisfied.

We next consider case (ii), so $\xi \in [x_j, x_{j+1}]$, 
$\eta \in [x_{j-1}, x_j]$ and $0 \le \xi - \eta \le ch$. It follows that
$\xi - x_j = c_1 h$ and $x_j - \eta = c_2 h$, where $c_1 \ge 0$,
$c_2 \ge 0$, and $c_1 + c_2 \le c <1$. As before, we need to show that
inequalities \eqref{9.5} and \eqref{9.6} are satisfied. Inequality \eqref{9.6}
takes the form
\begin{multline}
\label{9.9}
w^I(\eta) = \frac{\eta-x_{j-1}}{h} \wrm_j + \frac{x_j- \eta}{h} \wrm(x_{j-1})
\\
\le \exp([M-M_2]h)
\Big[\frac{\xi-x_{j}}{h} \wrm_{j+1} + \frac{x_{j+1}- \xi}{h} \wrm(x_{j})\Big],
\end{multline}
which is equivalent to
\begin{equation}
\label{9.9p}
(\eta-x_{j-1}) + (x_j- \eta) \frac{\wrm(x_{j-1})}{\wrm_j}
\le \exp([M-M_2]h)
\Big[(\xi-x_{j}) \frac{\wrm_{j+1}}{\wrm_j} + (x_{j+1}- \xi)\Big].
\end{equation}
Since $\wrm(x_{j-1})/\wrm_j \le \exp(Mh)$, 
$\wrm_{j+1}/\wrm_j \ge \exp(-Mh)$,
$x_j - \eta = c_2 h$ and $\xi - x_j = c_1 h$, \eqref{9.9p} will be satisfied
if
\begin{equation}
\label{9.9pp}
(1-c_2) + c_2 \exp(Mh) \le \exp([M-M_2]h)[c_1 \exp(-Mh) + (1-c_1)].
\end{equation}
Because $c_2 \le c-c_1$, we have
\begin{equation*}
(1-c_2) + c_2 \exp(Mh) \le (1-c+c_1) + (c-c_1) \exp(Mh),
\end{equation*}
and inequality \eqref{9.9pp} will be satisfied if
\begin{equation}
\label{9.10}
(1+ c_1 -c) + (c-c_1) \exp(Mh) \le \exp(-M_2 h)[c_1 + (1-c_1) \exp(Mh)].
\end{equation}
A necessary condition that \eqref{9.10} be satisfied is that
$\exp(-M_2 h) \ge (c-c_1)/(1-c_1)$.  Since $(c-c_1)/(1-c_1) \le c$ and
$c < (1+c)/2$, we choose $h = (b-a)/n >0$ sufficiently small that
\begin{equation}
\label{9.11}
\exp(-M_2 h) \ge (1+c)/2.
\end{equation}
For this choice of $h$, \eqref{9.10} will be satisfied if
\begin{equation*}
(1+ c_1 -c) + (c-c_1) \exp(Mh) \le \frac{1+c}{2}[c_1 + (1-c_1) \exp(Mh)],
\end{equation*}
which is equivalent to
\begin{equation}
\label{9.12}
(1 + c_1/2)(1-c) \le [(1+c_1)(1-c)/2] \exp(Mh).
\end{equation}
Since $(2+c_1)/(1+c_1) \le 2$, \eqref{9.12} will be satisfied if
\begin{equation}
\label{9.13}
2 \le \exp(Mh).
\end{equation}
Thus \eqref{9.9} will be satisfied if $h$ satisfies \eqref{9.11} and, for
this $h$, $M$ satisfies \eqref{9.13}.

Inequality \eqref{9.5} will be satisfied in case (ii) if
\begin{equation}
\label{9.14}
(\xi-x_j) \frac{\wrm_{j+1}}{\wrm_j} + (x_{j+1} - \xi)
\le \exp([M-M_2]h)
\Big[(\eta-x_{j-1}) + (x_{j}- \eta)
\frac{\wrm(x_{j-1})}{\wrm_j}\Big].
\end{equation}
The same reasoning as above shows that if $h >0$ satisfies \eqref{9.11}
and $M$ then satisfies \eqref{9.13}, \eqref{9.14} will be satisfied.
Details are left to the reader.
\end{proof}

\begin{thm}
\label{thm:9.3}
Let $N$ denote a positive integer. For $1 \le j \le N$, assume
that $\theta_j:[a,b] \to [a,b]$ is a Lipschitz map with $Lip(\theta_j) \le c
<1$, $c$ independent of $j$.  For $1 \le j \le N$, assume that $\beta_j \in
K_{M_0} \setminus \{0\} \subset X_n$, where $M_0$ is independent of $j$.  For
$j \ge 1$, let $C_j$ be a real number with $|C_j| \le C$, where $C$ is
independent of $j$; and for a fixed $s \ge 0$, define $\hat \beta_{j,s} \in
X_n$ by
\begin{equation*}
\hat \beta_{j,s}(x_k) = [1 + \tfrac{1}{2} C_j Q(\theta_j(x_k))] [\beta_j(x_k)]^s,
\quad  0 \le k \le n.
\end{equation*}
Let $\delta >0$ be a given real number and for $j \ge 1$ define
a linear map $L_{j,s}: X_n \to X_n$ by
\begin{equation*}
(L_{j,s} \wrm)(x_k) = \hat \beta_{j,s}(x_k) w^I(\theta_j(x_k)), 
\quad  0 \le k \le n,
\end{equation*}
and a linear map $L_s: X_n \to X_n$ by $L_s = \sum_{j=1}^N L_{j,s}$.  Assume
that $h = (b-a)/n \le 1$ and $C h /4 \le 1$ and define $M_2 = M_1
+ \delta$.  Assume also that $\exp(-M_2 h) \ge (1+c)/2$ and that $M \in \R$ is
such that $\exp(Mh) \ge 2$.  Then we have that $L_s(K_M \setminus \{0\})
\subset K_{M- \delta} \setminus \{0\}$.
\end{thm}
\begin{proof}
  Lemma~\ref{lem:9.1} implies that $x_k \mapsto \hat \beta_{j,s}(x_k)$ is an
  element of $K_{M_1}$, where $M_1 = s M_0 + (1+h)/2$.  Under our hypotheses,
  Lemma~\ref{lem:9.2} implies that $L_{j,s}(K_M \setminus \{0\}) \subset K_{M
    - \delta}\setminus \{0\}$, so $L_{s}(K_M \setminus \{0\}) \subset K_{M -
    \delta}\setminus \{0\}$.
\end{proof}

Our next theorem follows immediately from Theorem~\ref{thm:9.3} and the
remarks at the beginning of this section.

\begin{thm}
\label{thm:9.7}
Let notation and assumptions be as in Theorem~\ref{9.3}.  Then $L_s$ has an
eigenfunction $v \in K_{M-\delta} \setminus \{0\}$, $\|v\|=1$, with eigenvalue
$r >0$.  If $\hat L_s$ denotes the complexification of $L_s$, $r$ is an
eigenvalue of $\hat L_s$ of algebraic multiplicity one; and if $L_s w =
\lambda w$ for some $w \in K_M \setminus \{0\}$, $\lambda =r$, and $w$ is a
positive multiple of $v$.  If $z$ is an eigenvalue of $\hat L_s$ and $z \neq r$,
then $|z| <r$. If $x \in K_M\setminus\{0\}$, $\lim_{k \rightarrow \infty}
\|L^k(x)/\|L^k(x)\| -v\| =0$ and the convergence rate is geometric.
\end{thm}

\begin{remark}
\label{rem:9.7}
With the aid of Theorem~\ref{thm:9.3}, we could also have used the
theory of $u_0$-positive linear operators (see \cite{X} and \cite{Y})
to derive Theorem~\ref{thm:9.7}.
\end{remark}

\begin{remark}
\label{rem:9.8}
Since the linear maps $A_s$ and $B_s$ are both of the form of the map
$L_s$ in Theorem~\ref{thm:9.3}, Theorem~\ref{thm:9.7} implies the
desired spectral properties of $A_s$ and $B_s$.  With greater care it is
possible to use results in \cite{U} to estimate the spectral
clearance $q(L_s)$ of $L_s$.
\end{remark}

\begin{remark}
\label{rem:9.9}
We claim that there is a constant $E$, which can be easily estimated, such
that, for $h = (b-a)/n$ sufficiently small,
\begin{equation*}
r(B_s) \le r(A_s)(1 + E h^2).
\end{equation*}
(Of course we already know that $r(A_s) \le r(B_s)$.) For a fixed $s \ge 0$,
let $\beta_j(\cdot)$ and $\theta_j(\cdot)$ be as in Theorem~\ref{thm:9.3}.
We know that $A_s$ and $B_s$ are of the form of $L_s$ in
Theorem~\ref{thm:9.3}, so we can write, for $0 \le k \le n$,
\begin{align*}
(A_s \wrm)(x_k) &= \sum_{j=1}^N [ 1 +(C_j/2) Q(\theta_j(x_k))] [\beta_j(x_k)]^s
w^I(\theta_j(x_k),
\\
(B_s \wrm)(x_k) &= \sum_{j=1}^N [ 1 +(D_j/2) Q(\theta_j(x_k))] [\beta_j(x_k)]^s
w^I(\theta_j(x_k).
\end{align*}
We assume that $h \le 1$ and $C h/4 \le 1$, where $C$ is a positive
constant such that $\max(|C_j|, |D_j|) \le C$ for $1 \le j \le N$. We
assume also that for $1 \le j \le N$, $C_j \le D_j$.  Let
$K= \{ \wrm \in X_n \, | \, \wrm(x_k) \ge 0 \text{ for } 0 \le k \le n\}$, so
$A_s(K) \subset K$ and $B_s(K) \subset K$.  Define $\mu \ge 1$ by
\begin{equation*}
\mu = \sup\{ [1 + \frac{D_j}{2} Q(\theta_j(x_k))][ 1 + \frac{C_j}{2}
Q(\theta_j(x_k))]^{-1}: 1 \le j \le N, 0 \le k \le N\} \ge 1.
\end{equation*}
Then for all $\wrm \in K$ and $0 \le k \le n$,
$(B_s(\wrm))(x_k) \le \mu (A_s(\wrm))(x_k)$, which implies that
$r(B_s) \le \mu r(A_s)$.  Since $Q(u) \le h^2/4$, a little thought shows
that  $\mu \le (1 + Ch^2/8)(1 - Ch^2/8)^{-1} \le 1 + E h^2$,
which gives the desired estimate.
\end{remark}

\section{Log convexity of the spectral radius of 
$\Lambda_s$}
\label{sec:logconvex}
Throughout this section we shall assume that hypotheses (H4.1), (H4.2), and
(H4.3) in Section~\ref{sec:exist} are satisfied and we shall also assume that
$H$ is a bounded, open, subset of $\R$. As in Section~\ref{sec:exist}, we
shall write $X= C^m(\bar H)$ and $Y= C(\bar H)$.  For $s \in \R$, we define
$\Lambda_s: X \to X$ and $L_s:Y \to Y$ by
\begin{align}
\label{2.1}
(\Lambda_s(w))(x) &= \sum_{\be \in \B} [g_{\be}(x)]^s w (\theta_{\be}(x)),
\\
\label{2.2}
(L_s(w))(x) &= \sum_{\be \in \B} [g_{\be}(x)]^s w (\theta_{\be}(x)).
\end{align}
Theorem~\ref{thm:1.1} implies that $r(\Lambda_s)$ is an algebraically simple
eigenvalue of $\Lambda_s$ for $s \in \R$ and that $\sup \{|z|: z \in
\sigma(\Lambda_s), z \neq r(\Lambda_s)\} < r(\Lambda_s)$, 
where $\sigma(\Lambda_s)$ denotes the spectrum of $\Lambda_s$.

Let $\hat X$ denote the complexification of $X$, so $\hat X$ is the Banach
space of $C^m$ maps $f:H \to \C$ such that $x \mapsto
(D^{k}f)(x)$ extends continuously to $\bar H$ for all 
$0 \le k \le m$. For $s \in \C$ one can define
$\hat \Lambda_s: \hat X \to \hat X$ by
\begin{equation*}
(\hat \Lambda_s(w))(x) = \sum_{\be \in \B} (g_{\be}(x))^s w (\theta
_{\be}(x)):= \sum_{\be \in \B} \exp(s \log g_{\be}(x))   w(\theta_{\be}(x)).
\end{equation*}
The reader can verify that $s \mapsto \hat \Lambda_s \in \Lc(\hat X, \hat
X)$ is an analytic map. Because $r(\hat \Lambda_s)$ is an algebraically simple
eigenvalue of $\hat \Lambda_s$ for $s \in \R$ and $\sup \{|z|: z \in
\sigma(\Lambda_s), z \neq r(\Lambda_s)\} < r(\Lambda_s)$, it follows from the
kind of argument used on pages 227-228 of \cite{BB} that there is an open
neighborhood $U$ of $\R$ in $\C$ and the map $s \in U \mapsto r(\hat
\Lambda_s)$ is analytic on $U$.
\begin{thm}
\label{thm:2.1}
Assume that hypotheses (H4.1), (H4.2), and (H4.3) are satisfied with $m \ge 1$
and that $H \subset \R$ is a bounded, open set. For $s \in
\R$, let $\Lambda_s$ and $L_s$ be defined by \eqref{2.1} and \eqref{2.2}.
Then we have that $s \mapsto r(\Lambda_s)$ is log convex, i.e., $s
\mapsto log(r(\Lambda_s))$ is convex on $[0,\infty)$.
\end{thm}
\begin{proof}
Because Theorem~\ref{thm:1.1} implies that $r(L_s) = r(\Lambda_s)$ for all real
$s$, it suffices to take $s_0 <s_1$, and $0 < t < 1$ and prove that
\begin{equation*}
r(L_{(1-t)s_0 + t s_1}) \le r(L_{s_0})^{1-t} r(L_{s_1})^t.
\end{equation*}
We shall use an old trick (see \cite{CC} and the references therein). Let
$v_{s_j}(x)$, $j=0,1$ denote the strictly positive eigenfunction of $L_{s_j}$
which is ensured by Theorem~\ref{thm:1.1}.  Then
$L_{s_j} v_{s_j} = r(L_{s_j}) v_{s_j}$.
For a fixed $t$, $0<t<1$, define $s_t = (1-t)s_0 + t s_1$ and
\begin{equation*}
w_t(x)  = [v_{s_0}(x)]^{1-t}[v_{s_1}(x)]^t.
\end{equation*}
Then, using H\"older's inequality, we find that
\begin{multline}
\label{2.4}
(L_{s_t}(w_t))(x) = \sum_{\be \in \B} [g_{\be}(x)^{s_0} v_{s_0}(x)]^{1-t}
 [g_{\be}(x)^{s_1} v_{s_1}(x)]^t
\\
\le \Big(\sum_{\be \in \B} g_{\be}(x)^{s_0} v_{s_0}(x)\Big)^{1-t}
\Big(\sum_{\be \in \B}  g_{\be}(x)^{s_1} v_{s_1}(x)\Big)^t
= [r(L_{s_0})^{1-t}r(L_{s_1})^t] w_t(x).
\end{multline}
Because $w_t(x) >0$ for all $x \in \bar H$, a standard argument (see Lemma 5.9
in \cite{N-P-L}) shows that
\begin{equation}
\label{2.5}
r(L_{s_t}) = \lim_{k \rightarrow \infty} \|L_{s_t}^k\|^{1/k} 
 = \lim_{k \rightarrow \infty} \|L_{s_t}^k(w_t)\|^{1/k} .
\end{equation}
Using inequalities \eqref{2.4} and \eqref{2.5}, we see that
$r(L_{s_t}) \le r(L_{s_0})^{1-t} r(L_{s_1})^t$.
\end{proof}

In general, if $V$ is a convex subset of a vector space $X$, we shall call
a map $f:V \to [0, \infty)$ log convex if (i) $f(x) = 0$ for all $x \in V$
or (ii) $f(x) >0$ for all $x \in V$ and $x \mapsto \log(f(x))$ is convex.
Products of log convex functions are log convex, and H\"olders inequality
implies that sums of log convex functions are log convex.

Results related to Theorem~\ref{thm:2.1} can be found in \cite{CC},
\cite{DD}, \cite{EE}, \cite{FF}, \cite{GG}, and \cite{HH}. Note that
the terminology {\it super convexity} is used to denote log convexity
in \cite{DD} and \cite{EE}, presumably because any log convex function
is convex, but not conversely.  Theorem~\ref{thm:2.1}, while adequate for
our immediate purposes, can be greatly generalized by a different argument
that does not require existence of strictly positive eigenfunctions.  This
generalization (which we omit) contains Kingman's matrix log convexity
result in \cite{EE} as a special case.

In our applications, the map $s \mapsto r(L_s)$ will usually be strictly
decreasing on an interval $[s_1,s_2]$ with $r(L_{s_1}) >1$ and
$r(L_{s_2}) <1$, and we wish to find the unique $s_* \in (s_1,s_2)$ such that
$r(L_{s_*}) =1$.  The following hypothesis ensures that $s \mapsto r(L_s)$
is strictly decreasing for all $S$.

\noindent (H8.1): Assume that $g_{\be}(\cdot)$, $\be \in \B$ satisfy the
conditions of (H4.1).  Assume also that there exists an integer $\mu \ge 1$
such that $g_\w(x) <1$ for all $\w \in \B_{\mu}$ and all $x \in \bar H$.

\begin{thm}
\label{thm:2.2}
Assume hypotheses (H4.1), (H4.2), (H4.3), and (H8.1) are satisfied. Then the
map $s \mapsto r(\Lambda_s)$, $s \in \R$, is strictly decreasing and real
analytic and $\lim_{s \rightarrow \infty} r(\Lambda_s) =0$.
\end{thm}
\begin{proof}
  If $L_s: C(\bar H) \to C(\bar H)$ is given by \eqref{1.2}, it is a standard
  result that $r(L_s^{\nu}) = (r(L_s))^{\nu}$ and $r(\Lambda_s^{\nu}) =
  (r(\Lambda_s))^{\nu}$ for all integers $\nu \ge 1$, and
  Theorem~\ref{thm:1.1} implies that $r(L_s) = r(\Lambda_s)$.  Thus it
  suffices to prove that for some positive integer $\nu$, $s \mapsto
  r(L_s^{\nu})$ is strictly decreasing and $\lim_{s \rightarrow \infty}
  r(L_s^{\nu}) =0$.

Suppose that $K$ denotes the set of nonnegative functions in $C(\bar H)$ and
$A: C(\bar H) \to C(\bar H)$ is a bounded linear map such that $A(K)
\subset K$.  If there exists $w \in C(\bar H)$ such that $w(x) > 0$ for
all $x \in \bar H$ and if $(A(w))(x) \le a w(x)$ for all $x \in \bar H$,
it is well-known (and easy to verify) that $r(A) \le a$, where $r(A)$
denotes the spectral radius of $A$.  In our situation, we take $\nu = \mu$,
where $\mu$ is as in (H8.1), and $A = (L_s)^{\mu}$. If $s <t$ and $v_s$ is the
strictly positive eigenfunction for $(L_s)^{\mu}$, (H8.1) implies that
there is a constant $c <1$, $c = c(s,t)$, such that $c g_{\w}(x)^s \ge
g_\w(x)^t$ for all $\w \in \B_{\mu}$ and $x \in H$.  Thus we find that
\begin{equation*}
c r(L_s)^{\mu} v_s(x) = \sum_{\w \in \B_{\mu}} c g_{\w}(x)^s v_s(\theta_\w(x))
\ge \sum_{\w \in \B_{\mu}} g_{\w}(x)^t v_s(\theta_\w(x)) = (L_t^{\mu}(v_s))(x).
\end{equation*}
It follows that $r(L_t)^{\mu} \le c(s,t) r(L_s)^{\mu}$, so $r(L_t) < r(L_s)$,
for $s < t$.  Because $0 < g_\w(x) <1$ for all $x \in \bar H$ and $\w \in
\B_{\mu}$, it is also easy to see that $\lim_{t \rightarrow \infty}
\|(L_t)^{\mu}\| =0$; and since $\|(L_t)^{\mu}\| \ge r(L_t^{\mu})$, we see that
$\lim_{t \rightarrow \infty} r(L_t^{\mu}) =0$.
\end{proof}

\begin{remark}
\label{rem:2.3}
It is easy to construct examples for which (H8.1) is satisfied for some
$\mu >1$, but not satisfied for $\mu =1$.  The functions $\theta_1(x):=
9/(x+1)$ and $\theta_2(x):= 1/(x+2)$ both map the closed interval $\bar H
= [1/11, 9]$ into itself. There is a unique nonempty compact set $J \subset
\bar H$ such that
$J = \theta_1(J) \cup \theta_2(J)$.
For $s \in \R$, define $L_s:C(\bar H) \to C(\bar H)$ by
\begin{equation*}
(L_sw)(x):= \sum_{j=1}^2 |D \theta_j(x)|^s w(\theta_j(x)):=
\sum_{j=1}^2 g_j(x)^s w(\theta_j(x)),
\end{equation*}
where $D:= d/dx$.  The Hausdorff dimension of $J$ is the unique $s = s_*$,
$0 < s_* <1$, such that $r(L_s) =1$.  Our previous remarks show that
\begin{equation*}
(L_s^2 w)(x) = \sum_{j=1}^2 \sum_{k=1}^2 |D(\theta_j \circ \theta_k)(x)|^s
w(\theta_j \circ \theta_k)(x)).
\end{equation*}
One can check that (H8.1) is not satisfied for $\mu=1$, but is satisfied
for $\mu =2$.
\end{remark}
\begin{remark}
\label{rem:2.4}
Assume that the assumptions of Theorem~\ref{thm:2.2} are satisfied and define
$\psi(x) = \log(r(L_s)) = \log(r(\Lambda_s))$ (where $\log$ denotes the
natural logarithm), so $s \mapsto \psi(s)$ is a convex, strictly
decreasing function with $\psi(0) >1$ (unless $|\B| = p =1$) and
$\lim_{s \rightarrow \infty} \psi(s) = - \infty$. We are interested in finding
the unique value of $s$ such that $\psi(s) =0$. In general suppose that
$\psi:[s_1,s_2] \to \R$ is a continuous, strictly decreasing, convex
function such that $\psi(s_1) >0$ and $\psi(s_2) <0$, so there exists
a unique $s =s_* \in (s_1,s_2)$ with  $\psi(s_*) =0$.  If $t_1$ and $t_2$ are
chosen so that $s_1 \le t_1 < t_2 \le s_*$ and $t_{k+1}$ is obtained from
$t_{k-1}$ and $t_k$ by the secant method, an elementary argument show that
$\lim_{k \rightarrow \infty} t_k = s_*$.  If $s_* \le t_2 < t_1 < s_2$ and
$s_1 \le t_3$, a similar argument shows that $\lim_{k \rightarrow \infty} t_k
= s_*$. If $\psi \in C^3$, elementary numerical analysis implies
that the rate of convergence is faster than linear ($= (1 + \sqrt{5})/2)$.
In our numerical work, we apply these observations, not directly to $\psi(s) =
\log(r(\Lambda_s))$, but to convex decreasing functions which closely
approximate $\log(r(\Lambda_s))$.
\end{remark}

One can also ask whether the maps $s \mapsto r(B_s)$ and $s \mapsto r(A_s)$
are log convex, where $A_s$ and $B_s$ are the previously described
approximating matrices for $L_s$.  An easier question is whether the map
$s \mapsto r(M_s)$ is log convex, where $A_s$ and $B_s$ are obtained from
$M_s$ by adding error correction terms.  We shall prove that $s \mapsto
r(M_s)$ is log convex.

First, we need to recall a useful theorem of Kingman \cite{EE}.  Let
$M(s) = (a_{ij}(s))$ be an $m \times m$ matrix whose entries $a_{ij}(s)$ are
either strictly positive for all $s$ in a fixed interval $J$ or are
identically zero for all $s \in J$. Assume that $s \mapsto a_{ij}(s)$ is
log convex on $J$ for $1 \le i,j \le m$.  Under these assumptions,
Kingman \cite{EE} has proved that $s \mapsto r(M_s)$ is log convex.

Let $n \ge 2$ be a positive integer, and for $a < b$ given real numbers,
define $x_k = a + kh$, $-1 \le k \le n+1$, $h = (b-a)/n$. Let $X_n$
denote the vector space of real valued maps $\wrm:\{x_k \, | \, 0 \le k \le n\}
\to \R$, so $X_n$ is a real vector space linearly isomorphic to $\R^{n+1}$.
As usual, if $\wrm \in X_n$, extend $\wrm$ to a map $w^I:[a,b] \to \R$ by
linear interpolation, so
\begin{equation*}
w^I(u) = \frac{u-x_k}{h} \wrm(x_{k+1}) + \frac{x_{k+1}-u}{h} \wrm(x_{k}), \qquad
x_k \le u \le x_{k+1}, \quad 0 \le k \le n.
\end{equation*}
For $1 \le j \le N$, assume that $\theta_j:[a,b] \to [a,b]$ are given
maps and assume that $g_j:[a,b] \to (0, \infty)$ are given positive
functions.  For $s \in \R$, define a linear map $M_s:X_n \to X_n$ by
\begin{equation*}
M_s \wrm(x_k) = \sum_{j=1}^N [g_j(x_k)]^s f^I(\theta_j(x_k)), \quad 0 \le k \le n,
\end{equation*}
so if $\wrm(x_k) \ge 0$ for $0 \le k \le n$, $g(x_k) \ge 0$ for $0 \le k \le n$.
We can write $M_s \wrm(x_k) = \sum_{m=0}^n a_{km}(s) \wrm(x_m)$, where for
$0 \le k$, $m \le n$,
\begin{multline*}
a_{km}(s) = \sum_{j, x_{m-1} \le \theta_j(x_k) \le x_m} [g_j(x_k)]^s
[\theta_j(x_k) - x_{m-1}]/h
\\
+  \sum_{j, x_{m} \le \theta_j(x_k) \le x_{m+1}} [g_j(x_k)]^s
[x_{m+1} - \theta_j(x_k)]/h.
\end{multline*}
If, for a given $k$ and $m$, there is no $j$, $1 \le j \le N$, with
$ x_{m-1} \le \theta_j(x_k) \le x_{m+1}$, we define $a_{km} =0$.
Since the sum of log convex functions is log convex, $s \mapsto a_{km}(s)$ is
log convex on $\R$.  It follows from Kingman's theorem that $s \mapsto r(M_s)$
is log convex, where $r(M_s)$ denotes the spectral radius of $M_s$.

\end{document}